\documentclass{scrartcl}

\usepackage{lipsum}
\usepackage{amsfonts}
\usepackage{graphicx}
\usepackage{epstopdf}
\usepackage{algorithmic}
\ifpdf
  \DeclareGraphicsExtensions{.eps,.pdf,.png,.jpg}
\else
  \DeclareGraphicsExtensions{.eps}
\fi
\usepackage{amssymb}
\usepackage{dsfont}
\usepackage{bm}
\usepackage{xspace}

\newcommand{\diff}{\mathrm{d}}
\newcommand{\Dist}{\mathcal{D}}

\newcommand{\Jop}{J}
\newcommand{\metric}{g_{\Jop}}
\newcommand{\Reeb}{R}

\newcommand{\ReebT}{\widetilde{R}}
\newcommand{\OmegaRm}{\Omega_{\textnormal{Rm}}}
\newcommand{\diffRm}{\mathrm{d}_{\textnormal{Rm}}}
\newcommand{\OmegaH}{\Omega_{\textnormal{H}}}
\newcommand{\classRm}[1]{[#1]_{\textnormal{Rm}}}
\newcommand{\equivRm}{\overset{\textnormal{Rm}}{\sim}}
\newcommand{\deltaRm}{\delta_{\textnormal{Rm}}}
\newcommand{\diffH}{\diff_{\textnormal{H}}}

\newcommand{\diver}{\textnormal{div}}

\newcommand{\DeltaRm}{\Delta_{\textnormal{Rm}}}
\newcommand{\nablaRm}{\nabla_{\textnormal{Rm}}}
\newcommand{\norm}[1]{\| #1 \|_{\Jop}}
\newcommand{\horiz}[1]{#1_{\textnormal{H}}}
\newcommand{\DeltaRmA}[1]{\Delta_{\textnormal{Rm},#1}}
\newcommand{\nablaRmA}[1]{\nabla_{\textnormal{Rm},#1}}
\newcommand{\scalar}[2]{\langle #1,#2\rangle_{\Jop}}
\newcommand{\doubleNablaRm}{\nabla^2_{\textnormal{Rm}}}
\newcommand{\Wfirst}{W^1_{\textnormal{Rm}}}
\newcommand{\normWfirst}[1]{\| #1 \|_{\Wfirst}}
\newcommand{\Wsecond}{W^2_{\textnormal{Rm}}}
\newcommand{\normWsecond}[1]{\| #1 \|_{\Wsecond}}
\newcommand{\LH}{\mathfrak{L}_{\textnormal{H}}}
\newcommand{\LHB}{\mathfrak{L}_{\textnormal{HB}}}

\newcommand{\metricP}{g_{\Jop_{\mathbb{T}^2_k}}}

\newcommand{\metricB}{g_{\Jop_B}}
\newcommand{\scalarB}[2]{\langle #1,#2\rangle_{\Jop_B}}
\newcommand{\normB}[1]{\| #1 \|_{J_B}}
\newcommand{\Lfrak}{\mathfrak{L}}

\newcommand{\firstRev}[1]
{\textcolor{black}{#1}}

\title{Spectral properties of magnetic fields on sub-Riemannian contact manifolds
\thanks{Submitted for review on \today.}}

\author{Riccardo Bonalli\thanks{Laboratoire des Signaux et Syst\`emes, Université Paris-Saclay, CNRS, CentraleSupélec, Bât. Bréguet, 3 Rue Joliot Curie, 91190 Gif-sur-Yvette, France. \textit{E-mail}: {\texttt riccardo.bonalli@cnrs.fr}.}
\and Dario Prandi\thanks{Laboratoire des Signaux et Syst\`emes, Université Paris-Saclay, CNRS, CentraleSupélec, Bât. Bréguet, 3 Rue Joliot Curie, 91190 Gif-sur-Yvette, France. \textit{E-mail}: {\texttt dario.prandi@cnrs.fr}.}}

\usepackage{amsopn}

\usepackage{todonotes}

\usepackage{mathtools}
\mathtoolsset{showonlyrefs}

\usepackage{xcolor}
\usepackage{amsmath,amsthm,amssymb,hyperref}
\newtheorem{theorem}{Theorem}[section]
\newtheorem{lemma}{Lemma}[section]
\newtheorem{proposition}{Proposition}[section]
\newtheorem{corollary}{Corollary}[section]
\theoremstyle{definition}
\newtheorem{remark}{Remark}[section]

\ifpdf
\hypersetup{
  pdftitle={Title},
  pdfauthor={R. Bonalli and D. Prandi}
}
\fi

\begin{document}

\maketitle

\begin{abstract}
Motivated by some recent studies of the magnetic Laplacian on Riemannian manifolds, we focus on the first eigenvalue of the magnetic horizontal Laplacian on contact manifolds. We characterize conditions for positive spectral shift, and provide some sharp upper bounds. In the Riemannian setting, a genus 1 assumption is known to force the underlying metric to be flat when equality holds in the sharp upper bounds. Interestingly, we show that the equivalent topological condition in the three--dimensional contact setting consists of having first Betti number equal to 2.
In this case, equality in our upper bounds implies that the structure is that of a Heisenberg left--invariant nilmanifold.
%
We conclude by showing that, in some specific three--dimensional contact settings, the knowledge of the first eigenvalue of the magnetic Laplacian uniquely determines the manifold Chern class, fully determining the topology of the underlying manifold.
\end{abstract}




\section{Introduction}
\label{sec:Introduction}

The concept of magnetic fields on Riemannian manifolds underlines the deep connections between geometry and physics, and has been deeply investigated both from the classical and the quantum point of view. 
Given a Riemannian manifold $M$, we will mostly consider magnetic fields to be (complex) 2--forms $B\in \Omega^2(M)$ which are closed (i.e., $\diff B=0$) and exact. That is, $\diff B=0$, and there exists a one form $A\in \Omega^1(M)$, the magnetic potential, such that $\diff A=B$. In this formalism, the quantum dynamics are driven by the magnetic Laplacian, that is the operator defined by
\begin{equation}
    \Delta_A \triangleq -\star(\diff + iA\wedge)\star(\diff + iA\wedge).
\end{equation}
A key result for the quantum point of view is the diamagnetic inequality, which 
especially implies that magnetic potentials yield spectral improvements, whose nature depends on many factors but mainly on the fact that the underlying manifold be compact or not.

These improvements appear even in case the magnetic potential generates a null magnetic field, which can only happen in the non--simply connected case.
This is the so--called Aharonov--Bohm phenomenon, first highlighted in \cite{ehrenbergRefractive1949,aharonovSignificance1959} and henceforth experimentally validated in \cite{peshkinAharonovBohm1989}, which is a manifestation of the fact that, unlike to what happens for classical trajectories, in quantum mechanics the choice of the potential, and not only of the magnetic field, significantly affect the dynamics.

When the Riemannian manifold under consideration is the one dimensional sphere, i.e., $M=\mathbb{S}^1$, this phenomenon is apparent. Indeed, considering the magnetic potential $A = 2\pi\alpha\, \diff \theta$, with $\alpha\in \mathbb R$, a Fourier--based argument easily shows that the first eigenvalue $\lambda_1(A)$ of the corresponding magnetic Laplacian satisfies
\begin{equation}
    \label{eq:improvement-sphere}
    \lambda_1(A) = \operatorname{dist}(\alpha,\mathbb{Z})^2.
\end{equation}
In \cite{laptevHardy2011}, the authors leveraged this fact to prove that a non-integer Aharanov--Bohm potential  $A = 2\pi\alpha({x_1 \diff x_2 - x_2 \diff x_1})/{|x|}$ on the punctured plane $\mathbb{R}^2\setminus\{0\}$ yields a sub--critical magnetic Laplacian. More precisely, the following sharp Hardy inequality is shown to hold true in the sense of quadratic forms:
\begin{equation}
    \label{eq:laptev}
    \Delta_A \ge \frac{\operatorname{dist}(\alpha,\mathbb{Z})^2}{|x|^2}.
\end{equation}

The spectral improvement \eqref{eq:improvement-sphere} has been partially generalized to the case of a compact Riemannian manifold $M$ endowed with a magnetic potential $A \in \Omega^1 (M)$ in \cite{shigekawaEigenvalue1987}, where the author shows that
\begin{equation}
    \label{eq:shigekawa}
    \lambda_1(A) = 0 \quad \iff \quad \diff A =0 \text{ \ and } \int_\gamma A \in 2\pi\mathbb{Z} \text{ \ for any smooth loop } \gamma: \mathbb S^1 \to M .
\end{equation}
This result has been complemented in \cite{Colbois2017} by a quantitative upper bound which, for Aharonov--Bohm magnetic potentials (that is, $A$ harmonic and thus $\diff A = 0$) reads
\begin{equation}
    \label{eq:colbois-savo}
    \lambda_1(A)\le
    \frac{\operatorname{dist}(A,\mathfrak{L}^1(M))^2}{\operatorname{vol}(M)},
\end{equation}
where $\mathfrak L^1(M)$ is the lattice of harmonic forms $\omega$ with $\int_\gamma \omega \in 2\pi \mathbb{Z}$ for any loop $\gamma:\mathbb{S}^1\to M$. 
The authors also show a rigidity result: if the above inequality is an equality and $M$ is two--dimensional with genus $1$, i.e., it is diffeomorphic to a torus, then $M$ must be the flat torus, see \cite[Theorem 4.2]{Colbois2017}. For general genus surfaces or higher dimensional tori, it is possible to identify a countable family of Aharonov--Bohm potential (the magnetic ground state spectrum) such that the corresponding first eigenvalues uniquely determine the metric, for instance see the result in \cite{colboisMagnetic2025}.

Recently, building upon the works of Rumin \cite{Rumin1994}, in \cite{Franceschi2020} the concept of magnetic field has been extended to the setting of the Heisenberg group, which is the most basic example of sub--Riemannian contact manifold. The authors also generalize various spectral improvements to this setting, and in particular they have showed that a sub--Riemannian generalization of \eqref{eq:laptev} holds. Since the Rumin complex is well-defined for any contact manifold, this naturally begs the question of whether appropriate sub--Riemannian generalizations of \eqref{eq:shigekawa} and \eqref{eq:colbois-savo} hold similarly in the contact case. We point out that in the recent work \cite{Barilari2025} the authors have investigated similar generalizations through geodesic--based techniques. Yet, their approach, which does not leverage spectral analysis, applies to classical, i.e., non--quantum, settings.

A sub--Riemannian contact manifold is an odd-dimensional manifold $M$ endowed with a contact $1$--form $\alpha$, i.e., a $\alpha\in \Omega^1(M)$ such that $\alpha\wedge (\diff\alpha)^n\neq 0$, and a metric $h$ on the distribution $\Dist = \ker\alpha$. In this work, we will always assume this metric to be given by $h(X,Y)=\diff \alpha(X,JY)$, where $J:TM\to TM$ is any \emph{almost--complex structure} (see Section~\ref{sec:Preliminaries}). An important difference w.r.t.~the Heisenberg case considered in \cite{Franceschi2020}, is that in the general contact case there is no global frame for the metric $h$. It is still possible, however, to associate to any form $A\in \Omega^1(M)$ its horizontal representative $\horiz{A}\in \OmegaH^1(M)$ by requiring that $\horiz{A}\wedge (\diff\alpha)^n\equiv 0$. This allows one to define the horizontal differential $\diffH:\Omega^0(M)\to \OmegaH^1(M)$ as a projection $\diffH \varphi=\horiz{(\diff \varphi)}$ which, in turn, leads to defining the magnetic horizontal Laplacian (acting on smooth functions and associated with the magnetic potential $A$) through the rule
\begin{equation}
    \DeltaRmA{A} \triangleq -\star(\diff+iA\wedge)\star (\diffH + i\horiz{A}\wedge).
\end{equation}
When $M$ is compact, this operator is essentially self--adjoint on $L^2(M)$, non--negative, and with discrete spectrum. In particular, its first eigenvalue $\lambda_1(A)\ge 0$, our main focus, is well-defined.
Our first result is then the following generalization of \eqref{eq:shigekawa}.

\begin{theorem} 
It holds that $\lambda_1(A) = 0$ iff 
$\diffRm \horiz{A} = 0$ and
\begin{equation*}
\begin{aligned}
\int_{\gamma} A \in 2\pi \mathbb{Z}, \quad &\text{for every absolutely continuous (AC),} \\
&\text{closed, horizontal curve }\gamma:[0,1]\to M \text{ (i.e., such that $\dot \gamma(t)\in \Dist(\gamma(t))$)}.
\end{aligned}
\end{equation*}
\end{theorem}
Here, $\diffRm$ is the Rumin differential (see Section~\ref{sec:RuminCohomolAbdDiff}). We stress that the technique of proof we employ to obtain the above result is novel and adapts verbatim to the Riemannian setting, yielding an alternative proof to the one in \cite{shigekawaEigenvalue1987}. See Section~\ref{sec:spectrum-shift}.

When the magnetic potential is of Aharonov--Bohm type, i.e., it is Rumin--harmonic, thus generally weaker than harmonic, we have the following generalization of \eqref{eq:colbois-savo}.

\begin{theorem}
Assume that $\horiz{A}\in \OmegaH^1(M)$ is Rumin--harmonic, then it holds that 
\begin{equation}
    \label{eq:bound-intro}
    \lambda_1(A)\le \frac{\operatorname{dist}(\horiz{A},\LH(M))^2}{\operatorname{vol}(M)}.
\end{equation}
Here, $\LH(M)$ is the lattice of Rumin--harmonic $1$--forms $\omega \in \OmegaH^1(M)$ with $\int_\gamma \omega \in 2\pi\mathbb{Z}$, for any AC loop $\gamma:\mathbb{S}^1\to M$ that is horizontal.
\end{theorem}

As is the case for \eqref{eq:colbois-savo}, the above result is a consequence of a more general theorem, i.e., Theorem~\ref{theo:bounds}, which holds for any magnetic potential $A$ and relies on a Hodge--type decomposition theorem extended to the Rumin complex. 

Our final result is a rigidity result for three--dimensional $K$--contact regular manifolds. These are contact manifolds such that the Reeb vector field is Killing, and all its orbits are closed with the same period. 
Therefore, our result is the following.

\begin{theorem} \label{theo:HeisenbergMetricIntro}
Let $M$ be a three--dimensional $K$--contact regular manifold with first Betti number $b_1(M)=2$. Then, $M$ is contactomorphic to an Heisenberg nilmanifold $\Gamma_k\setminus\mathbb{H}^1$, $k\in\mathbb{N}^*$, endowed with a metric $h$. Moreover, equality holds in bound \eqref{eq:bound-intro} if and only if $h$ is induced by the Heisenberg left--invariant metric (see Section \ref{sec:Heisenberg} for a definition).
\end{theorem}

Observe that the first part of this theorem is a consequence of the Boothby--Wang Theorem \cite{Geiges2008}. Importantly, unlike what happens in the purely Riemannian setting of \cite[Theorem 4.2]{Colbois2017} where the topology is completely determined by the genus 1 assumption, for contact manifolds the assumption on the Betti number does not fix the topology of $M$. In particular, there is a priori no knowledge of the Chern class $k$. Interestingly, as a consequence of Theorem \ref{theo:HeisenbergMetricIntro} and some explicit eigenvalue computations on Heisenberg nilmanifolds, this missing information can be fully determined from if $\lambda_1(A)$ is known for a sufficiently rich class of magnetic potentials $A$. 

\begin{corollary}
    Let $M$ be a three--dimensional K--contact regular manifold with contact 1--form $\alpha_M$ and first Betti number $b_1(M) = 2$. In addition, let $A^1, A^2 \in \Omega^1(M)$ be linearly independent, Rumin--harmonic, and such that equality holds in bound \eqref{eq:bound-intro} for either $A^1$ or $A^2$. 
    Then, there exists $\gamma_1,\gamma_2\in\mathbb{R}$ such that the Chern class $k \in \mathbb{N}\setminus\{0\}$ of $M$ is uniquely determined by the sequence $\{ \lambda_1( ( \gamma_1 A^1 + \gamma_2 A^2 ) / \ell ) : \ell \in \mathbb{N}^* \}$.
\end{corollary}

\section{Preliminary Results}
\label{sec:Preliminaries}
\subsection{Contact, K--contact, and Sasakian manifolds}


The definitions and results below are taken from \cite{Geiges2008,Blair2010}. In this work, $M$ denotes a (real and smooth) compact and connected manifold of dimension $2 n + 1$, $n \in \mathbb{N}$. Forms and vector fields 
with values respectively in $\mathbb{C}$ and $T^{\mathbb{C}} M \simeq TM \oplus i TM$ are considered.

From now on, we assume the existence of a (real) 1--form $\alpha \in \Omega^1(M)$, called \textit{contact 1--form}, such that $\alpha \wedge (\diff \alpha)^n \neq 0$. 
This endows $M$ with a co--orientable contact structure and grants the existence of a unique vector field $\Reeb \in \Gamma(T M)$, called \textit{Reeb vector field}, such that $\alpha(\Reeb) = 1$ and $\diff \alpha(\Reeb,\cdot) = 0$. We denote by $\Dist \triangleq \text{ker} \ \alpha$ the distribution generated by $\alpha$.
We endow $M$ with a sub--Riemannian metric, i.e., a metric $h$ on the distribution $\Dist$. This defines a (1,1)--tensor\footnote{Throughout the manuscript, we implicitly identify bundle endomorphisms $E : T M \to T M$ with (1,1)--tensors through the isomorphism $E \mapsto T_E(\omega,X) \triangleq \omega(E X)$.} $J:TM\to TM$ by requiring that $JR=0$ and $h(v,w) = \diff \alpha(v,J w)$ for every $v , w \in \Dist$. We require $J$ to be an \textit{almost--complex structure}, i.e., additionally satisfying $J^2 v = -v$ for every $v \in \Dist$, so that $\diff \alpha(v , J w) = -\diff \alpha(J v , w)$ and $\diff \alpha(v , J v) > 0$ for every $v , w \in \Dist$. Since $J$ completely determines $h$, we henceforth refer to the sub--Riemannian contact structure on $M$ as being determined by $(\alpha,J)$.

Although the existence of $\alpha$ entails that at least one almost--contact structure exists, the choice of $\Jop$ is not unique. Any such $\Jop$ enables equipping $M$ with the \textit{adapted Riemannian metric} $\metric(\cdot,\cdot) \triangleq \alpha \otimes \alpha(\cdot,\cdot) + \diff \alpha(\cdot,\Jop \cdot)$, where the induced norm is denoted by $\norm{\cdot}$. For this reason, the tuple $(M,\alpha,\Jop)$ is called \textit{contact metric manifold}, henceforth referred to as \textit{contact manifold} for simplicity. Levi--Civita connections and musical isomorphisms will be computed with respect to $\metric$. A contact manifold is called \textit{K--contact} if its Reeb is Killing. A contact manifold $(M,\alpha,\Jop)$ is a \textit{CR--manifold} if the complex sub-bundle $\{ v - i \Jop v : v \in \Dist \} \subseteq T^{\mathbb{C}} M$ is involutive. CR--manifolds that are K-contact are called \textit{Sasakian}. In the rest of this work, $M$ denotes a generic contact manifold, implicitly equipped with a contact 1--form $\alpha$ and some almost--contact structure $\Jop$. Although $\alpha$ is assumed fixed, in this work $\Jop$ serves as a degree of freedom in the choice of 
$\metric$.

With no confusion arising, we still denote by $\metric$ the corresponding Hermitian metric on $T^{\mathbb{C}} M$. We denote by $\scalar{\cdot}{\cdot}: \Omega^{\bullet}(M) \times \Omega^{\bullet}(M) \to C^{\infty}(M,\mathbb{C})$ the scalar product on the external algebra of complex forms $\Omega^{\bullet}(M)$ induced by $\metric$, satisfying
$$
\scalar{\omega}{\eta} = \sum_{1 \le i_1 < \dots < i_k \le 2 n + 1} \omega(E_{i_1},\dots,E_{i_k}) \ \overline{\eta(E_{i_1},\dots,E_{i_k})} , \quad \omega , \eta \in \Omega^k(M) ,
$$
for any local $\metric$--orthonormal basis $\{ E_1 , \dots , E_{2 n + 1} \}$ of $T M$, thus $\scalar{\omega}{\eta} = \metric( \omega^{\sharp} , \eta^{\sharp} )$ in $\Omega^1(M)$. With no confusion arising, $\norm{\cdot}$ denotes the norm associated with $\scalar{\cdot}{\cdot}$.

Two kinds of local basis for $T M$ that are adapted to the contact structure are particularly useful. First, let $z$, $x_j$, and $y_j$, be local Darboux coordinates for $j=1,\dots,n$, for which it holds that $\Reeb = \partial_z$ and $\alpha = \diff z - \sum^n_{j=1} y_j \diff x_j$. By denoting $E_j \triangleq \partial_{x_j} + y_j \partial_z$ and $E_{j+n} \triangleq \partial_{y_j}$ for $j=1,\dots,n$, it holds that $\{ E_1 , \dots , E_{2 n} \}$ is a local basis for $\Dist$ such that $\diff \alpha = \sum^n_{j=1} E^*_j \wedge E^*_{j+n}$. By construction we have $[\Reeb , E_j] = 0$ and $R(E^*_j(R)) = 0$, for $j=1,\dots,2n$. 
We call $\{ R , E_1 , \dots , E_{2n} \}$ a \textit{Darboux frame}.
Second, it is always possible to select local $\metric$--orthonormal basis $\{ E_1 , \dots , E_{2 n} \}$ of $\Dist$ that are adapted to the almost–contact structure, i.e., satisfying $\Jop E_j = E_{j+n}$ and $\Jop E_{j+1} = -E_j$ for $j=1,\dots,n$, so that $\{ \Reeb , E_1 , \dots , E_{2 n} \}$ is a local $\metric$--orthonormal basis of $T M$. We call $\{ \Reeb , E_1 , \dots , E_{2 n} \}$ an \textit{almost--contact frame}. Note that $\diff \alpha = \sum^n_{j=1} E^*_j \wedge E^*_{j+n}$. 

Any cover $\pi : \widetilde M \to M$ of a contact manifold is a contact manifold.

\begin{proposition} \label{Prop:contactCovering}
    Let $\pi : \widetilde M \to M$ be a cover of $M$. Equipped with $\widetilde \alpha \triangleq \pi^*\alpha$ as contact 1--form, $\widetilde M$ becomes a contact manifold (associated with some $\widetilde \Jop$). The corresponding Reeb $\ReebT \in \Gamma(\widetilde M)$ satisfies $\Reeb(\pi(p)) = \diff \pi(p) \cdot \ReebT(p)$, for every $p \in \widetilde M$.
\end{proposition}

\begin{proof}
    The first claim readily follows from $\widetilde \alpha \wedge (\diff \widetilde \alpha)^n = \pi^* \big( \diff \alpha \wedge (\diff \alpha)^n \big)$, whereas the second claim is straightforward consequence of
    \begin{align*}
        \diff \alpha(\pi(p))( \diff \pi(p) \cdot \widetilde \Reeb(p) , v ) &= \diff \widetilde \alpha(p)( \widetilde \Reeb(p) , \diff \pi(p)^{-1} \cdot v ) \\
        &= 0 = \diff \alpha(\pi(p))( \Reeb(\pi(p)) , v ) , \quad p \in \widetilde M , \ v \in T_{\pi(p)} M .
    \end{align*}
\end{proof}

\noindent An \textit{Absolutely Continuous (AC)} curve $\gamma : [0,1] \to M$ is \textit{horizontal} if $\dot{\gamma}(t) \in \Dist(\gamma(t))$.

\begin{remark} \label{remark:horizComplete}
    Any couple of points $p , q \in M$ can always be joined by an AC horizontal curve, denoted $\gamma_{p,q} : [0,1] \to M$, as a consequence of Chow--Rashevskii's theorem \cite{Agrachev2003}. This can be seen by arguing locally, selecting, e.g., an almost--contact frame $\{ \Reeb , E_1 , \dots , E_{2 n} \}$ defined on a neighborhood $U \subseteq M$. For every $p , q \in U$, the objective boils down to showing the following control problem has 
    a solution:
    $$
    \begin{array}{c}
        \textnormal{Find $u \in L^2([0,1],\mathbb{R}^{2 n})$ such that the solution $\gamma : [0,1] \to M$} \\
        \textnormal{to $\displaystyle \dot{\gamma}(t) = \sum^{2 n}_{j=1} u_j(t) E_j(\gamma(t))$ satisfies $\gamma(0) = p$ and $\gamma(1) = q$.}
    \end{array}
    $$
    But $T M|_U = \textnormal{Lie}(E_1,\dots,E_{2 n})$, since $[E_i,E_{i+n}] \notin \Dist$, $i=1,\dots,2n$, and we conclude.
\end{remark}

\subsection{The horizontal differential}

We denote
$$
\OmegaH^k(M) \triangleq \Big\{ \omega \in \Omega^k(M) : \ \iota_{\Reeb} \omega = 0 \Big\} \subseteq \Omega^k(M)
$$
the sub-space of \textit{horizontal $k$--forms}, whereas we denote $\Jop \omega \triangleq \omega \circ \Jop^{(k)}$ for $\omega \in \Omega^k(M)$, where $\Jop^{(k)} \triangleq J \otimes \dots \otimes J$ is the $k$--tensor product. In particular, it holds that $\Jop \omega = (\Jop \omega^{\sharp})^{\flat}$ for $\omega \in \Omega^1(M)$. Any $\omega \in \OmegaH^k(M)$ satisfying $\mathcal{L}_{\Reeb} \omega = 0$ is called \textit{basic}. Let us denote $\Pi(v) \triangleq v - \alpha(v) \Reeb \in \Dist$, $v \in TM$, and $\Pi^{(k)} \triangleq \Pi \otimes \dots \otimes \Pi$ the $k$--tensor product, and define $\horiz{\omega} \triangleq \omega \circ \Pi^{(k)} \in \OmegaH^k(M)$ for $\omega \in \Omega^k(M)$. The identity $\horiz{\omega} = \omega - \alpha \wedge \iota_{\Reeb} \omega$ is readily seen to hold by induction, due to exterior product identities for 1--forms:
\begin{align*}
    \horiz{\omega}(X_1,\dots,X_k) &= ( \iota_{X_1} \omega - \alpha(X_1) \iota_{\Reeb} \omega ) \circ \Pi^{(k-1)}(X_2,\dots,X_k) \\
    &= \omega(X_1,\dots,X_k) - ( \alpha \wedge \iota_{\Reeb} \iota_{X_1} \omega + \alpha(X_1) \iota_{\Reeb} \omega )(X_2,\dots,X_k) \\
    &= ( \omega - \alpha \wedge \iota_{\Reeb} \omega )(X_1,\dots,X_k) , \qquad \textnormal{for} \qquad X_1 , \dots , X_k \in \Gamma(T M) .
\end{align*}

We define the \textit{horizontal differential} as $\diffH \omega \triangleq \horiz{(\diff \omega)} \in \OmegaH^{k+1}(M)$ for $\omega \in \Omega^k(M)$.

\begin{lemma} \label{lemma:zeroHorizDiff}
    For every $\omega \in \Omega^k(M)$, it holds that $\diff \omega = 0$ iff $\diffH \omega = 0$.
\end{lemma}

\begin{proof}
    By restricting our analysis to a local neighborhood, let $\{ \Reeb , E_1 , \dots , E_{2 n} \}$ be, e.g., an almost--contact frame and assume that $\diffH \omega = 0$. Differentiating and evaluating $\diffH \omega = \diff \omega - \alpha \wedge \iota_{\Reeb} \diff \omega$ along $\{ E_i , E_{i+n} , E_{i_1} , \dots , E_{i_{k-1}} \}$, for $i , i+n \neq i_1 , \dots , i_{k-1}$, yields $\diff \omega(\Reeb , E_{i_1} , \dots , E_{i_{k-1}}) = \iota_{\Reeb} \diff \omega(E_{i_1} , \dots , E_{i_{k-1}}) = 0$, and the conclusion follows.
\end{proof}

\subsection{A primer on Rumin cohomology} \label{sec:RuminCohomolAbdDiff}

We recall a few 
key concepts on the \textit{Rumin cohomology} \cite{Rumin1994}, a natural extension of the de Rham cohomology in the contact case. Since we are 
interested in real--valued 1--forms, for the sake of conciseness we only describe up to the second Rumin complex: 
\begin{equation} \label{eq:complexRumin}
    \OmegaRm^0(M) \overset{\diffRm}{\longrightarrow} \OmegaRm^1(M) \overset{\diffRm}{\longrightarrow} \OmegaRm^2(M) .
\end{equation}
Here, $\OmegaRm^0(M)$ coincides with the usual 0--forms, i.e., functions in $C^{\infty}(M,\mathbb{R})$, whereas $\OmegaRm^1(M)$ is made of equivalence classes $\classRm{\omega}$ stemming from the relation
$$
\omega_1 , \omega_2 \in \Omega^1(M) : \quad \omega_1 \equivRm \omega_2 \quad \Longleftrightarrow \quad \omega_2 = \omega_1 + f \alpha , \ f \in C^{\infty}(M,\mathbb{R}) .
$$
The operator $\diffRm : \OmegaRm^0(M) \to \OmegaRm^1(M)$ corresponds to the classical differential, i.e., $\diffRm \classRm{f} \triangleq \classRm{\diff f}$ for $\classRm{f} \in \OmegaRm^0(M)$. The main novelty of the Rumin cohomology is the operator $\diffRm : \OmegaRm^1(M) \to \OmegaRm^2(M)$, whose definition depends on $n$. Specifically, if $n > 1$, then $\OmegaRm^2(M)$ is made of equivalence classes $\classRm{\eta}$:
$$
\eta_1 , \eta_2 \in \Omega^2(M) : \ \eta_1 \equivRm \eta_2 \ \Longleftrightarrow \ \eta_2 = \eta_1 + \theta \wedge \alpha + f \diff \alpha , \ \theta \in \Omega^1(M) , \ f \in C^{\infty}(M,\mathbb{R}) ,
$$
in which case $\diffRm : \OmegaRm^1(M) \to \OmegaRm^2(M)$ corresponds to the classical differential, i.e., $\diffRm \classRm{\omega} \triangleq \classRm{\diff \omega}$ for $\classRm{\omega} \in \OmegaRm^1(M)$. However, in the case of $n = 1$, the space $\OmegaRm^2(M)$ does not contain equivalence classes but forms $\eta \in \Omega^2(M)$ such that $\eta \wedge \alpha = 0$, and the differential operator $\diffRm : \OmegaRm^1(M) \to \OmegaRm^2(M)$ is defined by $\diffRm \classRm{\omega} \triangleq \diff( \omega + f_{\omega} \alpha )$ for $\classRm{\omega} \in \OmegaRm^1(M)$, where $f_{\omega} \in C^{\infty}(M,\mathbb{R})$ is uniquely such that $\diff( \omega + f_{\omega} \alpha ) \in \OmegaRm^2(M)$. The complex \eqref{eq:complexRumin} naturally induces cohomological classes that are isomorphic to the classical de Rham cohomological classes of $M$.

The equivalence classes $\OmegaRm^1(M)$ and $\OmegaRm^2(M)$ (in the case of $n > 1$) can be identified to sub-spaces of ``true'' forms. For this, let us denote by
$$
L : \OmegaH^k(M) \to \OmegaH^{k+2}(M) : \omega \mapsto \diff \alpha \wedge \omega , \quad k \ge 0 ,
$$
the Lefschetz operator and by $\Lambda : \OmegaH^{k+2}(M) \to \OmegaH^k(M)$ its pointwise adjoint relative to the metric $\scalar{\cdot}{\cdot}(p)$, for every $p \in M$. The following identifications, where equivalence classes are replaced by specific representative, are thus natural:
$$
\OmegaRm^1(M) \sim \{ \horiz{\omega} \in \OmegaH^1(M) : \ \omega \in \Omega^1(M) \}
$$
and
$$
\OmegaRm^2(M) \sim \begin{cases}
    \left\{ \displaystyle \horiz{\eta} - \frac{L \Lambda \horiz{\eta}}{n} \in \OmegaH^2(M) \cap \text{ker} \ \Lambda : \ \eta \in\Omega^2(M) \right\} , \quad n > 1 , \\
    \hspace{0.5ex} \{ \eta \in \Omega^2(M) : \ \eta \wedge \alpha = 0 \} , \hspace{24.25ex} n = 1 .
\end{cases}
$$
In particular, we may also identify, for every $f \in \OmegaRm^0(M)$ and $\omega \in \OmegaRm^1(M)$,
$$
\diffRm \classRm{f} \sim \diffRm f \triangleq \diffH f  \quad \text{and} \quad \diffRm \classRm{\omega} \sim \diffRm \omega \triangleq \begin{cases}
    \displaystyle \diffH \omega - \frac{L \delta (\Jop \omega)}{n} , \ n > 1 , \\
    \diff( \omega - \delta (\Jop \omega) \alpha ) , \hspace{1.25ex} n = 1 ,
\end{cases}
$$
using the fact that $\Lambda \diffH \omega = \delta(\Jop \omega)$, whether $n > 1$ or $n = 1$ \cite{Rumin1994}. A quick, stand--alone proof of this latter identity goes as follows. First of all, note that $\nabla_{\Reeb} X$ is horizontal for every horizontal $X \in \Gamma(T M)$. Indeed, Cartan's formula yields $\mathcal{L}_{\Reeb} \alpha = 0$, hence $\alpha( [\Reeb,X] ) = -\mathcal{L}_{\Reeb} \alpha(X) = 0$. Since $\nabla_X \Reeb = -\Jop X - \Jop h_R(X)$ with $h_R = \mathcal{L}_{\Reeb} J / 2$ holds in any contact manifold, the sought claim follows from $\nabla_{\Reeb} X = \nabla_X \Reeb + [\Reeb , X]$. Next, by restricting our analysis to a local neighborhood, let $\{ \Reeb , E_1 , \dots , E_{2 n} \}$ be an almost--contact frame. We may conclude by computing
\begin{align*}
    -\delta \omega &= \sum^{2n}_{\ell=1} g_{\alpha}( \nabla_{E_{\ell}} \omega^{\sharp} , E_{\ell} ) + g_{\alpha}( \nabla_{\Reeb} \omega^{\sharp} , {\Reeb} ) = \sum^{2n}_{\ell=1} \diff \alpha( [E_{\ell} , \omega^{\sharp}] , \Jop E_{\ell} ) \\
    &= \sum^n_{\ell=1} \alpha( [E_{\ell+n} , [E_{\ell} , \omega^{\sharp}]] + [E_{\ell} , [\omega^{\sharp} , E_{\ell+n}]] ) - \sum^{2n}_{\ell=1} \Jop E_{\ell}(\alpha( [E_{\ell},\omega^{\sharp}])) \\
    &= \sum^n_{\ell=1} \Big( \omega^{\sharp}( \alpha([E_{\ell} , E_{\ell+n}]) ) - \diff \alpha( \omega^{\sharp} , [E_{\ell} , E_{\ell+n} ] ) \Big) - \sum^{2n}_{\ell=1} \Jop E_{\ell}(\alpha( [E_{\ell},\omega^{\sharp}])) \\
    &= \sum^n_{\ell=1} \Big( E_{\ell+n}( \diff \alpha( E_{\ell} , \omega^{\sharp}) ) - E_{\ell}( \diff \alpha( E_{\ell+n} , \omega^{\sharp} ) - \diff \alpha( \omega^{\sharp} , [E_{\ell} , E_{\ell+n}] ) \Big) \\
    &= \sum^n_{\ell=1} \Big( E_{\ell}( \Jop \omega ( E_{\ell+n} ) ) - E_{\ell+n}( \Jop \omega( E_{\ell} ) ) - \Jop \omega( [E_{\ell} , E_{\ell+n} ] ) \Big) 
    = \Lambda \diffH (\Jop \omega) .
\end{align*}
All the previous identifications, assumed to hold as equivalences throughout this work, enable introducing $\deltaRm : \OmegaRm^{k+1}(M) \to \OmegaRm^k(M)$, for $k = 0,1$, adjoint operator of $\diffRm$ with respect to the metric $(\cdot,\cdot)_{L^2}$ on $\Omega^{\bullet}(M)$ (induced by the metric $\scalar{\cdot}{\cdot}$). We say a horizontal 1--form $\omega \in \OmegaH^1(M)$ is \textit{Rumin--harmonic} if $\diffRm \omega = 0$ and $\deltaRm \omega = 0$.

\begin{lemma} \label{lemma:reductionDiffRm}
    Let $\omega \in \OmegaH^1(M)$. It holds that $\deltaRm \omega = \delta \omega$. 
    Moreover, $\diffRm \omega = 0$ iff $\displaystyle \diff( \omega - \delta( \Jop \omega ) \alpha / n ) = 0$. If in addition $\omega$ is basic, then $\diffRm \omega = 0$ iff $\diff \omega = 0$, therefore $\omega$ basic is Rumin--harmonic iff it is harmonic (in the classical sense).
\end{lemma}

\begin{proof}
    Whether $n > 1$ or $n = 1$, the first claim readily follows from
    $$
    (\deltaRm \omega , g)_{L^2} = (\omega , \diffRm g)_{L^2} = (\omega , \diff g)_{L^2} = (\delta \omega , g)_{L^2} , \quad g \in C^{\infty}(M,\mathbb{R}) .
    $$
    
    We move to the second claim, which is trivially true in case $n = 1$. Otherwise,
    $$
    \displaystyle \diffH\left( \omega - \frac{\delta( \Jop \omega )}{n} \alpha \right) = \diffH \omega - \frac{\Lambda \diffH \omega}{n} \diff \alpha = \diffRm \omega = 0 ,
    $$
    and the second claim for $n > 1$ follows from Lemma \ref{lemma:zeroHorizDiff}.

    Finally, assume that $\omega$ is additionally basic. Since in such setting Cartan's formula yields that $\diff \omega \in \OmegaH^2(M)$, if $\diffRm \omega = 0$, thanks to the second claim we infer that
    $$
    \frac{\diff \delta(J \omega)(E_{\ell})}{n} = \diff\left( \omega - \frac{\delta( \Jop \omega )}{n} \alpha \right)(R,E_{\ell}) = 0 , \quad \ell=1,\dots,2n ,
    $$
    for any, e.g., almost–contact frame $\{ \Reeb , E_1 , \dots , E_{2 n} \}$. Hence, Lemma \ref{lemma:zeroHorizDiff} yields that $\diff \delta(J \omega) = 0$. Given that $M$ is connected, $\delta(J \omega)$ is constant, and thus from the divergence Theorem we infer that $\delta(J \omega) = 0$, and the conclusion follows. 
\end{proof}

\begin{remark}
    In general, ``basicness'' of a form $\omega \in \OmegaH^1(M)$ is also necessary for $\diffRm \omega = 0$ to imply that $\diff \omega = 0$. Indeed, consider $M \triangleq \mathbb{R}^3$ non--compact equipped with $\alpha \triangleq \diff z - y \diff x$, so that $\Reeb = \partial_z$, and let $\omega \triangleq f \diff x \in \OmegaH^1(M)$ with $f \triangleq 2 x z - x^2 y \in C^{\infty}(M,\mathbb{R})$. By denoting $X \triangleq \partial_x + y \partial_z$ and $Y \triangleq \partial_y$, one shows that, although
    $$
    \diffRm \omega = -( X Y + \Reeb ) f \; \diff x \wedge \alpha - Y^2 f \; \diff y \wedge \alpha = 0 ,
    $$
    it holds that
    $$
    \diff \omega = - Y f \; \diff x \wedge \diff y - \Reeb f \; \diff x \wedge \alpha = x^2 \; \diff x \wedge \diff y - 2 x \; \diff x \wedge \alpha \neq 0 .
    $$
    Note that the previous counterexample holds true independently from the almost--contact structure $M$ may be equipped with.
\end{remark}

From \cite[Corollaire, page 290]{Rumin1994} we readily infer the following extension of the Hodge decomposition theorem for 1--form to the contact case:

\begin{proposition} \label{prop:deRhamDec}
    Let $\omega \in \OmegaH^1(M)$ be real--valued. There exist $\xi \in \OmegaH^1(M)$ Rumin--harmonic, and $f\in C^{\infty}(M,\mathbb{R})$ and $\eta \in \OmegaRm^2(M)$ such that
    $$
    \omega = \xi \oplus (\diffH f + \deltaRm \eta) , \quad \text{where} \quad \xi , \ \deltaRm \eta \in \OmegaH^1(M) \quad \text{are co--closed} ,
    $$
    and the direct sum is with respect to the metric $(\cdot,\cdot)_{L^2}$ on $\Omega^{\bullet}(M)$. If in addition it holds that $\diffRm \omega = 0$, then the above decomposition reduces to $\omega = \xi \oplus \diffH f$.
\end{proposition}

A particularly practical application of Proposition \ref{prop:deRhamDec} stems from considering K--contact manifolds and basic forms, in which Rumin--harmonic terms in the de Rham decomposition can essentially be replaced by harmonic (in the classical sense) terms:

\begin{corollary} \label{corol:deRhamDec}
    Assume $M$ is K--contact. If $\omega \in \OmegaH^1(M)$ is basic, then there exist $\xi \in \OmegaH^1(M)$ harmonic (in the classical sense) that is additionally basic, as well as $f \in C^{\infty}(M,\mathbb{R})$ and $\eta \in \OmegaRm^2(M)$ such that
    $$
    \omega = \xi \oplus (\diffH f + \deltaRm \eta) , \quad \text{where} \quad \deltaRm \eta \in \OmegaH^1(M) \quad \text{is co--closed} ,
    $$
    and the direct sum is with respect to the metric $(\cdot,\cdot)_{L^2}$ on $\Omega^{\bullet}(M)$. If in addition it holds that $\diffRm \omega = 0$, then the above decomposition reduces to $\omega = \xi \oplus \diffH f$.
\end{corollary}

\begin{proof}
    Let $\omega = \xi \oplus (\diffH f + \deltaRm \eta)$ as in Proposition \ref{prop:deRhamDec}. Since computing $\star$ through its very definition along a Darboux frame yields $\mathcal{L}_{\Reeb} \star \beta = \star \mathcal{L}_{\Reeb} \beta$ for $\beta \in \OmegaH^k(M)$, $\mathcal{L}_{\Reeb}$ commutes with $\delta$. Hence, from Lemma \ref{lemma:reductionDiffRm} we infer that $\deltaRm \mathcal{L}_{\Reeb} \xi = \mathcal{L}_{\Reeb} \deltaRm \xi = 0$. Moreover, since Cartan's formula yields both $\mathcal{L}_{\Reeb} \alpha = 0$ and $\mathcal{L}_{\Reeb} \diff \alpha = 0$, and in K-contact settings it holds that $\mathcal{L}_{\Reeb} \Jop^{(k)} \beta = \Jop^{(k)} \mathcal{L}_{\Reeb} \beta$ for $\beta \in \Omega^k(M)$, we may compute
    \begin{align*}
        \diff \hspace{-0.5ex} \left( \mathcal{L}_{\Reeb} \xi - \frac{\delta( \Jop \mathcal{L}_{\Reeb} \xi )}{n} \alpha \right) 
        \hspace{-0.55ex} = \hspace{-0.25ex} \mathcal{L}_{\Reeb} \diff \xi - \frac{\mathcal{L}_R \diff \delta( \Jop \xi )}{n} \wedge \alpha - \frac{\mathcal{L}_R \delta( \Jop \xi )}{n} \diff \alpha 
        = \mathcal{L}_{\Reeb} \diff \hspace{-0.5ex} \left( \xi - \frac{\delta( \Jop \xi )}{n} \alpha \right) ,
    \end{align*}
    that turns out to equal zero due to Lemma \ref{lemma:reductionDiffRm}. In particular, from the very same lemma we may conclude that $\mathcal{L}_{\Reeb} \xi$ is Rumin--harmonic ($\mathcal{L}_{\Reeb} \xi$ is clearly horizontal).

    At this step, on the one hand, $\mathcal{L}_{\Reeb} \alpha = 0$ readily yields that $\mathcal{L}_{\Reeb} \diffH f = \diffH \Reeb(f)$. On the other hand, in case $n > 1$, $\mathcal{L}_{\Reeb} \eta$ is clearly horizontal and it holds that $\eta \in \text{ker} \ \Lambda$ iff $\diff \alpha \wedge \star \eta = 0$. Since $\mathcal{L}_{\Reeb} \diff \alpha = 0$ yields $\diff \alpha \wedge \star \mathcal{L}_{\Reeb} \eta = \mathcal{L}_{\Reeb}( \diff \alpha \wedge \star \eta ) = 0$, we infer that $\mathcal{L}_{\Reeb} \eta \in \OmegaRm^2(M)$. This latter inclusion is even simpler to prove if $n = 1$.
    In particular, assuming first that $n = 1$, 
    the definition of $\deltaRm$ yields that
    $$
    \mathcal{L}_{\Reeb} \deltaRm \eta = -\star \diff ( \star \mathcal{L}_{\Reeb} \eta - \delta( J \star \mathcal{L}_{\Reeb} \eta ) \alpha ) = \deltaRm \mathcal{L}_{\Reeb} \eta .
    $$
    By leveraging the formal definition of $\deltaRm$ for $n > 1$ via $\star$ and $\diffRm$ \cite{Rumin1994}, this latter identity is easily proved to hold if $n > 1$ as well. We finally conclude that
    $$
    \mathcal{L}_{\Reeb} \xi \oplus ( \diffH \Reeb(f) + \deltaRm \mathcal{L}_{\Reeb} \eta ) = \mathcal{L}_{\Reeb} \omega = 0 ,
    $$
    with $\mathcal{L}_{\Reeb} \xi \in \OmegaH^1(M)$ Rumin--harmonic and $\mathcal{L}_{\Reeb} \eta \in \OmegaRm^2(M)$. In particular, we infer that $\xi$ is basic, hence harmonic due to Lemma \ref{lemma:reductionDiffRm}.
\end{proof}

\section{The Horizontal Magnetic Laplacian and its First Eigenvalue}
\label{sec:Laplacian}
\subsection{Definitions} \label{sec:DefLap} 

We define the \textit{horizontal (Rumin) Laplacian} on functions as
$$
\DeltaRm \varphi \triangleq \deltaRm \diffRm \varphi = \delta \diffH \varphi = -\star \diff \star \diffH \varphi , \quad \varphi \in C^{\infty}(M,\mathbb{C}) ,
$$
and the corresponding \textit{horizontal (Rumin) gradient} as
$$
\nablaRm \varphi \triangleq (\diffRm \varphi)^{\sharp} = (\diffH \varphi)^{\sharp} \in \Gamma(T M) , \quad \varphi \in C^{\infty}(M,\mathbb{C}) .
$$
Applying the divergence Theorem on the real and imaginary parts separately yields
\begin{equation} \label{eq:equivLaplacian}
    \int_M \DeltaRm \varphi \ \overline{\psi} = \int_M \metric( \nablaRm \varphi , \nabla \psi ) = \int_M \scalar{\diffH \varphi}{\diff \psi} , \quad \varphi , \psi \in C^{\infty}(M,\mathbb{C}) .
\end{equation}

For any fixed real--valued $A \in \Omega^1(M)$, we define the \textit{magnetic horizontal (Rumin) Laplacian (associated with the magnetic potential $A$)} on functions as
$$
\DeltaRmA{A} \varphi \triangleq -\star ( \diff + i A \wedge ) \star (\diffH + i \horiz{A}) \varphi , \quad \varphi \in C^{\infty}(M,\mathbb{C}) ,
$$
and the corresponding \textit{magnetic horizontal (Rumin) gradient} as
$$
\nablaRmA{A} \varphi \triangleq ((\diffH + i \horiz{A}) \varphi)^{\sharp} \in \Gamma(T M) , \quad \varphi \in C^{\infty}(M,\mathbb{C}) .
$$

\begin{lemma} \label{lemma:MagnLap}
    The following identities hold true for every $\varphi , \psi \in C^{\infty}(M,\mathbb{C})$:
    \begin{itemize}
        \item $\DeltaRmA{A} \varphi = \DeltaRm \varphi + \norm{\horiz{A}}^2 \varphi + i ( \varphi \delta \horiz{A} - 2 \scalar{\diffH \varphi}{A} )$,
        \item $\displaystyle \int_M \DeltaRmA{A} \varphi \ \overline{\psi} = \int_M \metric( \nablaRmA{A} \varphi , \nablaRmA{A} \psi )$.
    \end{itemize}
\end{lemma}

\begin{proof}
    Thanks to \eqref{eq:equivLaplacian}, we may compute
    \begin{align*}
        \DeltaRmA{A} \varphi &= 
        \DeltaRm \varphi - i \ \diver( \varphi \horiz{A}^{\sharp} ) - i \ \scalar{ \diffH \varphi }{A} + \scalar{A}{\horiz{A}} \varphi ,
    \end{align*}
    and the first identity follows since $\diver( \varphi \horiz{A}^{\sharp} ) = -\varphi \ \delta \horiz{A} + \scalar{\diff \varphi}{\horiz{A}}$.
    
    For the second identity, on the one hand, the first identity yields
    \begin{align} \label{eq:LapLapProof}
        \int_M \DeltaRmA{A} \varphi \ \overline{\psi} = \int_M\DeltaRm \varphi \ \overline{\psi} \ &+ \int_M \norm{\horiz{A}}^2 \varphi \, \overline{\psi} \\
        &- i \left( \int_M \diver(\horiz{A}^{\sharp}) \varphi \, \overline{\psi} + 2 \int_M \scalar{\diffH \varphi}{\horiz{A} \psi} \right) . \nonumber
    \end{align}
    On the other hand, thanks to
    $$
    \int_M \overline{\scalar{\diffH \psi}{\horiz{A} \varphi}} = -\int_M \diver(\horiz{A}^{\sharp}) \varphi \, \overline{\psi} - \int_M \scalar{\diff_H \varphi}{\horiz{A}} \overline{\psi} ,
    $$
    and \eqref{eq:equivLaplacian}, we may infer that
    \begin{align*}
        &\int_M \metric( \nablaRmA{A} \varphi , \nablaRmA{A} \psi ) = \\
        &= \int_M \DeltaRm \varphi \ \overline{\psi} + \int_M \norm{\horiz{A}}^2 \varphi \, \overline{\psi} - i \left( \int_M \diver(\horiz{A}^{\sharp}) \varphi \, \overline{\psi} + 2 \int_M \scalar{\diffH \varphi}{\horiz{A} \psi} \right) ,
    \end{align*}
    concluding the proof due to \eqref{eq:LapLapProof}.
\end{proof}

To define the first eigenvalue of $\DeltaRmA{A}$, we denote the (1,1)--tensor $\doubleNablaRm \varphi(\mu,X) \triangleq \mu \circ \Pi( \nabla_{\Pi(X)} \nablaRm \varphi )$, for 
$\varphi \in C^{\infty}(M,\mathbb{C})$, $\mu \in \Omega^1(M)$, $X \in \Gamma(T M)$, and the norms
$$
\normWfirst{\varphi} \triangleq \sqrt{\int_M \Big (|\varphi|^2 + \norm{\nablaRm \varphi}^2 \Big)} \ \ \textnormal{and} \ \ \normWsecond{\varphi} \triangleq \sqrt{\normWfirst{\varphi} + \int_M \norm{\doubleNablaRm \varphi}^2} ,
$$
and define the functional spaces
$$
\Wfirst(M,\mathbb{C}) \triangleq \overline{C^{\infty}(M,\mathbb{C})}^{\normWfirst{\cdot}} , \ \Wsecond(M,\mathbb{C}) \triangleq \overline{C^{\infty}(M,\mathbb{C})}^{\normWsecond{\cdot}} \subseteq L^2(M,\mathbb{C}) .
$$
It is known that $\Wfirst(M,\mathbb{C})$ is Hilbert \cite{Franceschi2020}. Since for $\varphi \in C^{\infty}(M,\mathbb{C})$ clearly $|\DeltaRm \varphi| \le \norm{\doubleNablaRm \varphi}$ pointwise, thanks to \eqref{eq:equivLaplacian} and Lemma \ref{lemma:MagnLap} both the horizontal Laplacian and the magnetic horizontal Laplacian can be extended to densely--defined, non--negative and symmetric linear operators
$$
\DeltaRm , \ \DeltaRmA{A} : \ \Wsecond(M,\mathbb{C}) \subseteq L^2(M,\mathbb{C}) \to L^2(M,\mathbb{C}) .
$$

We now define the \textit{first eigenvalue of the magnetic horizontal Laplacian} as
$$
\lambda_1(A) \triangleq \underset{\footnotesize \begin{array}{c}
    \varphi \in \Wsecond(M,\mathbb{C}) \\
    \| \varphi \|_{L^2} = 1
\end{array}}{\inf} (\DeltaRmA{A} \varphi , \varphi)_{L^2} = \underset{\footnotesize \begin{array}{c}
    \varphi \in \Wsecond(M,\mathbb{C}) \\
    \| \varphi \|_{L^2} = 1
\end{array}}{\inf} \| \nablaRmA{A} \varphi \|^2_{L^2} \ge 0 .
$$
We recall that, given a Hilbert space $\mathcal{H}$, and a densely--defined, non--negative and symmetric linear operator $T : D(T) \subseteq \mathcal{H} \to \mathcal{H}$, the Maximum Principle yields the existence of a minimum $\varphi \in D(T)$ with $\| \varphi \|_{\mathcal{H}} = 1$ to the minimization problem
$$
(T \varphi , \varphi)_{\mathcal{H}} = \min \ \sigma(T) = \underset{\footnotesize \begin{array}{c}
    \varphi \in D(T) \\
    \| \varphi \|_{\mathcal{H}} = 1
\end{array}}{\inf} (T \varphi , \varphi)_{\mathcal{H}} \ge 0 ,
$$
in case there exists $\lambda \in \mathbb{R}$ such that the range $\textnormal{Rng}(T + \lambda I) = \mathcal{H}$, and the spectrum $\sigma(T) \subseteq [0,+\infty)$ of $T$ is discrete. In particular, $\varphi$ is eigenvector of $T$ with eigenvalue $\lambda_1(T) \triangleq \min \ \sigma(T)$. From this remark, we readily infer the infimum that defines $\lambda_1(A)$ is attained at some eigenfunction $\varphi \in \Wsecond(M,\mathbb{C})$ with $\| \varphi \|_{L^2} = 1$ from the following

\begin{lemma} \label{lemma:LaxMilgram}
    There exists $\lambda > 0$ such that $\textnormal{Rng}( \DeltaRmA{A} + \lambda I ) = L^2(M,\mathbb{C})$. In addition, the spectrum $\sigma(\DeltaRmA{A}) \subseteq [0,+\infty)$ of $\DeltaRmA{A}$ is discrete.
\end{lemma}

\begin{proof}
    The proof leverages routine functional analysis techniques. We report it in the appendix for the sake of completeness.
\end{proof}

The magnetic horizontal Laplacian satisfies the \textit{Gauge invariance principle}.

\begin{lemma} \label{lemma:Gauge}
    For $f , h \in C^{\infty}(M,\mathbb{R})$, it holds that $\lambda_1(A + \diff f + h \alpha) = \lambda_1(A)$.
\end{lemma}

\begin{proof}
    It is sufficient to prove that the isometry
    $$
    \varphi \in L^2(M,\mathbb{C}) \mapsto T(\varphi) \triangleq e^{i f} \varphi \in L^2(M,\mathbb{C})
    $$
    satisfies $T( \Delta_{H,A+\diff f+h \alpha} (T^{-1} \varphi) ) = \DeltaRmA{A} \varphi$, for every $\varphi \in L^2(M,\mathbb{C})$. For this, thanks to Lemma \ref{lemma:MagnLap}, for every $\varphi \in L^2(M,\mathbb{C})$ we may compute
    \begin{align*}
        &T( \Delta_{H,A+\diff f+g \alpha} (e^{-i f} \varphi) ) = \\
        &=  T\Big( -\diver\big( e^{-i f} \nablaRm \varphi - i e^{-i f} \varphi \nablaRm f \big) \\
        &\hspace{15ex} + \scalar{\horiz{A} + \diffH f}{\horiz{A} + \diffH f} e^{-i f} \varphi - i e^{-i f} \varphi \ \diver( \horiz{A}^{\sharp} + \nablaRm f ) \\
        &\hspace{22ex} + 2 e^{-i f} \scalar{\diffH \varphi}{\horiz{A} + \diffH f} - 2 i e^{-i f} \varphi \ \scalar{\diffH f}{\horiz{A} + \diffH f} \Big) \\
        &=  T\Big( e^{-i f} \big( \DeltaRm \varphi - i \big( \varphi \diver( \horiz{A}^{\sharp} ) + 2 \scalar{\diffH \varphi}{\horiz{A}} \big) + \scalar{\horiz{A}}{\horiz{A}} \varphi \big) \Big) = \DeltaRmA{A} \varphi .
    \end{align*}
\end{proof}

\subsection{Spectrum shift via magnetic potentials}
\label{sec:spectrum-shift}

Akin to the Riemmanian setting \cite{Colbois2017}, positive shifts of the spectrum of the magnetic horizontal Laplacian can be nicely characterized through algebraic--differential conditions on the magnetic field $A$.

\begin{theorem} \label{theo:SpectrumShift}
    It holds that $\lambda_1(A) = 0$ iff 
    $\diffRm \horiz{A} = 0$ and
    \begin{equation}
    \label{eq:flux}
    \int_{\gamma} A \in 2\pi \mathbb Z , \ \textnormal{for every AC, closed, and horizontal curve} \ \gamma : [0,1] \to M .
    \end{equation}
\end{theorem}

The proof makes use of the following

\begin{lemma} \label{lemma:nonZeroEigen}
    If there exists $\varphi \in C^{\infty}(M,\mathbb{C})$, $\varphi \neq 0$, such that $\diffH \varphi + i \horiz{A} \varphi = 0$, then $\diffRm \horiz{A} = 0$. In particular, there exists $f \in C^{\infty}(M,\mathbb{R})$ such that $\diff( A + f \alpha ) = 0$.
\end{lemma}

\begin{proof}
    The identity $\diffH \varphi + i \horiz{A} \varphi = 0$ is equivalent to $\diffH \overline{\varphi} - i \horiz{A} \overline{\varphi} = 0$, therefore $\diffH |\varphi|^2 
    = 0$. In particular, Lemma \ref{lemma:zeroHorizDiff} yields that $|\varphi| \neq 0$ is constant. Through a differentiation with respect to $\diffRm$ we thus conclude that
    $$
    \varphi \; \diffRm \horiz{A} = \diffRm( \varphi \horiz{A} ) - \diffH \varphi \wedge \horiz{A} = i ( \diffRm^2 \varphi + \varphi \; \horiz{A} \wedge \horiz{A} ) = 0 .
    $$
    The last claim follows from Lemma \ref{lemma:reductionDiffRm}.
\end{proof}

\begin{proof}[Proof of Theorem \ref{theo:SpectrumShift}]
    We split the proof into two sections.

    \vspace{5pt}
    
    \noindent\textit{Necessity.} Let $\lambda_1(A) = 0$, with eigenfunction $\varphi \in \Wsecond(M,\mathbb{C})$, $\| \varphi \|_{L^2} = 1$, so that $\diffH \varphi + i \horiz{A} \varphi = 0$ almost everywhere on $M$. In particular, we infer that $\DeltaRmA{A} \varphi = 0$. Given any neighborhood $U \subseteq M$, let $\{ \Reeb , E_1, \dots , E_{2n} \}$ denote a, e.g., Darboux frame defined in $U$. Thanks to Lemma \ref{lemma:MagnLap}, we thus infer that (see also \cite{Franceschi2020})
    \begin{align*}
        \sum^{2n}_{j=1} \int_U \varphi &\overline{E^2_j \psi + (\diver(E_j) - 2 i \horiz{A}(E_j)) E_j \psi} = \\
        &= \int_U ( i \delta \horiz{A}  - \norm{\horiz{A}}^2 ) \varphi \overline{\psi} , \ \textnormal{for every} \ \psi \in C^{\infty}_c(U,\mathbb{C}) .
    \end{align*}
    Since $M$ is a contact manifold, this latter variational problem falls into the hypoelliptic setting of \cite[Theorems 16 and 18]{Rothschild1976}, hence we infer that $E_{j_1} \dots E_{j_k} \varphi \in L^2(U,\mathbb{C})$, for $j_1,\dots,j_k=1,\dots,2n$ and $k \in \mathbb{N}$. But locally $\Reeb \in \textnormal{span} \{ E_1 , \dots , E_{2n} , [E_{i_1} , E_{i_2}] , i_1 , i_2 = 1 , \dots , 2n \}$, therefore the classical theory of Sobolev spaces finally yield $\varphi \in C^{\infty}(M,\mathbb{C})$ and $\diffH \varphi + i \horiz{A} \varphi = 0$ everywhere on $M$. Moreover, as in the proof of Lemma \ref{lemma:nonZeroEigen}, one shows that $|\varphi| \neq 0$ is constant. Up to a rescaling, we may thus assume that $|\varphi| = 1$.

    Let $\pi : \widetilde M \to M$ be the universal cover of $M$, which is regular given that $\widetilde M$ is simply connected. We define $\psi \triangleq \varphi \circ \pi \in C^{\infty}(\widetilde M,\mathbb{C})$, for which $|\psi| = 1$. By denoting $\widetilde{\horiz{A}} \triangleq \pi^* \horiz{A}$, Proposition \ref{Prop:contactCovering} thus yields that $\widetilde{\horiz{A}} \in \OmegaH^1(\widetilde M)$, and together with $\diffH \varphi + i \horiz{A} \varphi = 0$, that $\diffH \psi + i \widetilde{\horiz{A}} \psi = 0$. Now, besides 
    $\diffRm \horiz{A} = 0$, from Lemma \ref{lemma:nonZeroEigen} we also infer the existence of 
    $h \in C^{\infty}(M,\mathbb{R})$ such that $\diff( \horiz{A} + h \alpha ) = 0$. Hence, by denoting $\widetilde h \triangleq h \circ \pi$, we obtain that $\widetilde{\horiz{A}} + \widetilde h \widetilde \alpha \in \Omega^1(\widetilde M)$ is closed. In particular, there exists $\widetilde f \in C^{\infty}(\widetilde M,\mathbb{R})$ such that $\widetilde{\horiz{A}} = \diffH \widetilde f$, so that it holds that $\diffH e^{-i \widetilde f} + i \widetilde{\horiz{A}} e^{-i \widetilde f} = 0$. Note that the equation $\diffH \cdot + i \widetilde{\horiz{A}} \ \cdot = 0$ can have at most one solution in $C^{\infty}(\widetilde M,\mathbb{C})$ of constant norm equal to 1, up to multiplication by $e^{i \theta}$, $\theta \in \mathbb{R}$. Indeed, if $\psi_1 , \psi_2 \in C^{\infty}(\widetilde M,\mathbb{C})$ are two such solutions, as in the proof of Lemma \ref{lemma:nonZeroEigen}, one shows that $\diff (\psi_1 \psi^{-1}_2) = 0$, which implies that $\psi_2 = \psi_1 e^{i \theta}$, for some $\theta \in \mathbb{R}$. 
    Hence, $\varphi \circ \pi = e^{i (\theta - \widetilde f)}$ for some $\theta \in \mathbb{R}$, implying that
    \begin{equation} \label{eq:fAut}
        \textnormal{For every} \ G \in \textnormal{Deck}(\pi) , \ \textnormal{there exists} \ k_G \in \mathbb{Z} : \ \widetilde f \circ G = \widetilde f + 2 \pi k_G .
    \end{equation}
    
    Let $\gamma : [0,1] \to M$ be an AC, closed, and horizontal curve. For $j=1,\dots,p$, let $[a_j,b_j] \subseteq [0,1]$ so that $[0,1] = \cup^p_{j=1} [a_j,b_j]$, $a_{j+1} = b_j$, and $\gamma([a_j,b_j]) \subseteq \pi(\widetilde U_j)$, where $\widetilde U_j \subseteq \widetilde M$ are neighborhood such that $\pi_j \triangleq \pi|_{\widetilde U_j} : \widetilde U_j \to \pi(\widetilde U_j)$ are diffeomorphisms. Thus,
    $$
    \int_{\gamma} A = \sum^p_{j=1} \int^{b_j}_{a_j} \horiz{A}(\gamma(t)) \cdot \dot{\gamma}(t) \; \mathrm{d}t
    = \sum^p_{j=1} \int^{b_j}_{a_j} \diff \widetilde f\big( \pi_j^{-1}(\gamma(t)) \big) \cdot \frac{\mathrm{d}}{\mathrm{d}t} \pi_j^{-1}(\gamma(t)) \; \mathrm{d}t ,
    $$
    given that each $\pi^{-1}_j \circ \gamma : [a_j,b_j] \to \widetilde M$ is horizontal due to Proposition \ref{Prop:contactCovering}, so that
    \begin{equation} \label{eq:IntegralAProof}
        \int_{\gamma} A = \widetilde f\big( \pi_p^{-1}(\gamma(1)) \big) - \widetilde f\big( \pi_1^{-1}(\gamma(0)) \big) + \sum^{p-1}_{j=1} \left( \widetilde f\big( \pi_j^{-1}(\gamma(b_j)) \big) - \widetilde f\big( \pi_{j+1}^{-1}(\gamma(a_{j+1})) \big) \right) .
    \end{equation}
    Since $\pi$ is regular, for every $j=1,\dots,p-1$ there exist $G_j , G_p \in \textnormal{Deck}(\pi)$ such that $\pi_j^{-1}(\gamma(b_j)) = G_j\big( \pi_{j+1}^{-1}(\gamma(a_{j+1})) \big)$ and $\pi_p^{-1}(\gamma(1)) = G_p\big( \pi_1^{-1}(\gamma(0)) \big)$. Therefore, due to \eqref{eq:fAut}, for every $j=1,\dots,p-1$ there hold $\widetilde f\big( \pi_j^{-1}(\gamma(b_j)) \big) - \widetilde f\big( \pi_{j+1}^{-1}(\gamma(a_{j+1})) \big) = 2 \pi k_j$ and $\widetilde f\big( \pi_p^{-1}(\gamma(1)) \big) - \widetilde f\big( \pi_1^{-1}(\gamma(0)) \big) = 2 \pi k_p$, for appropriate $k_j , k_p \in \mathbb{Z}$. By plugging these latter identities in \eqref{eq:IntegralAProof}, we conclude the proof of this first section.

    \vspace{10px}

    \noindent\textit{Sufficiency.} 
    Let $\diffRm \horiz{A} = 0$. Lemma \ref{lemma:reductionDiffRm} yields that $\diff ( A + h \alpha ) = 0$ for some function $h \in C^{\infty}(M,\mathbb{R})$, and therefore, as in the proof of the previous section, we may find a function $\widetilde f \in C^{\infty}(\widetilde M,\mathbb{R})$ such that $\widetilde{\horiz{A}} \triangleq \pi^* \horiz{A} = \diffH \widetilde f$.
    
    Thanks to Propositions \ref{Prop:contactCovering} and Remark \ref{remark:horizComplete}, given any $p \in \widetilde M$ and $G \in \textnormal{Deck}(\pi)$ we may find an AC and horizontal curve $\widetilde \gamma : [0,1] \to \widetilde M$, such that $\widetilde \gamma(0) = p$ and $\widetilde \gamma(1) = G(p)$. Moreover, since Proposition \ref{Prop:contactCovering} yields that $\gamma \triangleq \pi \circ \widetilde \gamma : [0,1] \to M$ is an AC and horizontal curve that is closed at $\pi(p) \in M$, we may compute
    $$
    \widetilde f(G(p)) - \widetilde f(p) = \int^1_0 \diffH \widetilde f(\widetilde \gamma(t)) \cdot \dot{\widetilde \gamma}(t) \; \mathrm{d}t = \int_{\gamma} A \in 2 \pi \mathbb{Z} ,
    $$
    from which, due to the arbitrariness of $p \in \widetilde M$ and $G \in \textnormal{Deck}(\pi)$, we infer that \eqref{eq:fAut} still holds true. In particular, by defining $\psi \triangleq e^{-i \widetilde f} \in C^{\infty}(\widetilde M,\mathbb{C})$ with $|\psi| = 1$, thanks to \eqref{eq:fAut} one sees that $\psi \circ G = \psi$ for every $G \in \textnormal{Deck}(\pi)$. Thanks to this latter property and the fact that $\pi$ is a regular cover, the following function
    \begin{align*}
        \varphi : \ &M \to \mathbb{C} \\
        &p \mapsto \psi\big( (\pi|_{U_p})^{-1}(p) \big) , \quad p \in \pi(U_p) ,
    \end{align*}
    where $U_p \subseteq \widetilde M$ is any neighborhood such that $\pi|_{U_p} : U_p \to \pi(U_p)$ is a diffeomorphism, is well-defined, smooth, and of constant norm equal to 1. Moreover, by leveraging Proposition \ref{Prop:contactCovering}, straightforward computations yield that $\diffH \varphi + i \horiz{A} \varphi = 0$, from which we infer that $\lambda_1(A) = 0$, finally concluding the proof of Theorem \ref{theo:SpectrumShift}.
\end{proof}

\subsection{Upper bounds for the first eigenvalue}

Akin to the Riemmanian setting \cite{Colbois2017}, we now aim at deriving upper bounds for the first eigenvalue $\lambda_1(A)$ of the magnetic horizontal Laplacian. For this, we adapt and extend methods from \cite{Colbois2017}. The following lemma will be needed.

\begin{lemma} \label{lemma:EigenFunc}
    Let $\mu \in \Omega^1(M)$ and $h , \varphi \in C^{\infty}(M,\mathbb{C})$ be such that $|\varphi|=1\big/\sqrt{\textnormal{vol}(M)}$ on $M$, $\DeltaRmA{\mu} \varphi = h \varphi$, and $\lambda_1(\mu) = (\DeltaRmA{\mu} \varphi , \varphi)_{L^2}$. Then, $\lambda_1(\mu) = h = \textnormal{constant}$.
\end{lemma}

\begin{proof}
    By assumption and due to Lemma \ref{lemma:MagnLap}, the quadratic form
    \begin{align*}
        Q : \ &\Wfirst(M,\mathbb{C}) \to [0,+\infty) \\
        &\psi \mapsto \int_M \norm{\nablaRmA{\mu} \psi}^2 - \lambda_1(\mu) \| \psi \|^2_{L^2}
    \end{align*}
    attains its minimum at $\varphi \in C^{\infty}(M,\mathbb{C})$, therefore $( Q(\varphi + \varepsilon \psi) - Q(\varphi) ) / \varepsilon \ge 0$, for every $\psi \in \Wfirst(M,\mathbb{C})$ and $\varepsilon > 0$. This latter inequality and Lemma \ref{lemma:MagnLap} readily yield that
    $$
    \int_M \Big( \DeltaRmA{\mu} \varphi - \lambda_1(\mu) \varphi \Big) \overline{\psi} + \int_M \Big( \DeltaRmA{\mu} \psi - \lambda_1(\mu) \psi \Big) \overline{\varphi} = 0 , \quad \psi \in \Wsecond(M,\mathbb{C}) .
    $$
    Through Lemma \ref{lemma:MagnLap}, we may further compute
    \begin{align*}
        2 \ &\textnormal{Re}\left( \int_M \Big( \DeltaRmA{\mu} \varphi - \lambda_1(\mu) \, \varphi \Big) \overline{\psi} \right) = \\
        &= \int_M \DeltaRmA{\mu} \varphi \, \overline{\psi} + \int_M \DeltaRmA{\mu} \psi \, \overline{\varphi} - \int_M \lambda_1(\mu) ( \varphi \, \overline{\psi} + \psi \, \overline{\varphi} ) = 0 , \quad \psi \in \Wsecond(M,\mathbb{C}) ,
    \end{align*}
    and plugging $\psi = \DeltaRmA{\mu} \varphi - \lambda_1(\mu) \varphi \in \Wsecond(M,\mathbb{C})$ in the latter identity yields that $\varphi$ is eigenfunction associated with $\lambda_1(\mu)$, therefore $\lambda_1(\mu) = h = \textnormal{constant}$.
\end{proof}

We have the following upper bounds for the first eigenvalue $\lambda_1(A)$.

\begin{theorem} \label{theo:bounds}
    Let $\horiz{A} = \xi \oplus (\diffH f + \deltaRm \eta)$ as in Proposition \ref{prop:deRhamDec}, and define
    \begin{align*}
        \LH^1(M) \triangleq \Big\{ \omega \in \OmegaH^1(M) : \ &\omega \textnormal{\ is Rumin--harmonic and} \int_{\gamma} \omega \in 2 \pi \mathbb{Z} , \\
        &\textnormal{for every AC, closed, and horizontal} \ \gamma : [0,1] \to M \Big\} .
    \end{align*}
    It holds that
    \begin{equation} \label{eq:newUpperBound}
        \lambda_1(A) \le \frac{\textnormal{dist}(\xi , \LH^1(M))^2 + \| \deltaRm \eta \|^2_{L^2}}{\textnormal{vol}(M)} .
    \end{equation}
    In case $\diffRm \horiz{A} = 0$, bound \eqref{eq:newUpperBound} reduces to
    \begin{equation} \label{eq:newShortUpperBound}
        \lambda_1(A) \le \frac{\textnormal{dist}(\xi , \LH^1(M))^2}{\textnormal{vol}(M)} .
    \end{equation}
    Moreover, if equality holds in bound \eqref{eq:newShortUpperBound}, then there exists $\omega \in \LH^1(M)$ such that
    \begin{equation} \label{eq:constantForm}
        \lambda_1(A) = \norm{\xi - \omega}^2 = \textnormal{constant} , \quad \textnormal{pointwise on} \ M.
    \end{equation}
    Finally, in case $M$ is K--contact and $\horiz{A}$ is basic, let $\horiz{A} = \xi \oplus (\diffH f + \deltaRm \eta)$ as in Corollary \ref{corol:deRhamDec}. All the previous statements hold true with $\LH^1(M)$ replaced by
    \begin{align*}
        \LHB^1(M) \triangleq \Big\{ \omega \in \OmegaH^1(M) : \ &\omega \textnormal{\ is harmonic and basic, and} \int_{\gamma} \omega \in 2 \pi \mathbb{Z} , \\
        &\textnormal{for every AC, closed, and horizontal} \ \gamma : [0,1] \to M \Big\} .
    \end{align*}
\end{theorem}

\begin{proof}
    Since \eqref{eq:newShortUpperBound} follows from \eqref{eq:newUpperBound} and Proposition \ref{prop:deRhamDec}, and the last claim of Theorem \ref{theo:bounds} follows from Corollary \ref{corol:deRhamDec}, we just focus on proving \eqref{eq:newUpperBound} and \eqref{eq:constantForm}.

    To prove \eqref{eq:newUpperBound}, fix $\omega \in \LH^1(M)$ and $p_0 \in M$, and define $h(p) \triangleq \int_{\gamma_{p_0,p}} \omega$, where $\gamma_{p_0,p} : [0,1] \to M$ is any AC and horizontal curve that joins $p_0$ with $p \in M$, and whose existence is granted by Remark \ref{remark:horizComplete}. Since $\int_{\gamma} \omega \in 2 \pi \mathbb Z$ for every AC, closed, and horizontal curve $\gamma : [0,1] \to M$, the values of $h$ are uniquely defined mod $2 \pi$ for each $p \in M$. Now, fix $p \in M$ and let $\{ E_1 , \dots , E_{2n} , E_{2n+1} \triangleq R \}$ be any, e.g., Darboux frame defined in a neighborhood of $p$. The mapping $\psi(t_1,\dots,t_{2n+1}) \triangleq \big(\theta^{E_1}_{t_1} \circ \dots \circ \theta^{E_{2n+1}}_{t_{2n+1}}\big)(p)$, where $\theta^X$ denotes the flow of any $X \in \Gamma(T M)$, is a local diffeomorphism. 
    Therefore,
    \begin{align*}
        \frac{h \circ \psi(\varepsilon e_j) - h(p)}{\varepsilon} &= \frac{1}{\varepsilon} \int_{t \in [0,1] \mapsto \theta^{E_j}_{t \varepsilon}(p)} \omega \\
        &= \int^1_0 \omega( \theta^{E_j}_{t \varepsilon}(p) ) \cdot E_j( \theta^{E_j}_{t \varepsilon}(p) ) \; \mathrm{d}t , \quad j=1,\dots,2n ,
    \end{align*}
    which converges to $\omega(p) \cdot E_j(p)$ for $\varepsilon \to 0$, for each $j=1,\dots,2n$. It follows that the function $\varphi \triangleq e^{-i h}$ is well-defined and satisfies $\diffH \varphi + i \omega \varphi = 0$, from which one additionally infers that $\varphi \in C^{\infty}(M,\mathbb{C})$ by replicating the proof of Theorem \ref{theo:SpectrumShift}.

    At this step, thanks to Lemmas \ref{lemma:reductionDiffRm} and \ref{lemma:MagnLap}, for any $\mu \in \Omega^1(M)$ we may compute
    \begin{equation} \label{eq:DeltaForPhiProof}
        \DeltaRmA{\mu} \varphi = i \scalar{\diff \varphi}{\omega} + \norm{\horiz{\mu}}^2 \varphi - 2 \scalar{\omega}{\horiz{\mu}} \varphi + i \delta \horiz{\mu} \varphi = ( \norm{\horiz{\mu} - \omega}^2 + i \delta \horiz{\mu} ) \varphi .
    \end{equation}
    Combining this latter identity with Lemma \ref{lemma:Gauge} and the fact that $|\varphi| = 1$ finally yields
    \begin{align*}
        \lambda_1(A) &= \lambda_1(\horiz{A} - \diffH f) \le \frac{(\DeltaRmA{\xi \oplus \deltaRm \eta} \varphi , \varphi)_{L^2}}{\textnormal{vol}(M)} = \frac{\| \xi - \omega \|^2_{L^2} + \| \deltaRm \eta \|^2_{L^2}}{\textnormal{vol}(M)} .
    \end{align*}

    To prove \eqref{eq:constantForm}, let equality holds in \eqref{eq:newShortUpperBound}. By leveraging the latter computations, Lemma \ref{lemma:Gauge}, and the fact that $|\varphi|=1$, we infer the existence of $\omega \in \LH^1(M)$ such that
    $$
    \lambda_1(\xi) = \lambda_1(A) = \frac{\| \xi - \omega \|^2_{L^2}}{\textnormal{vol}(M)} = \left( \DeltaRmA{\xi} \frac{\varphi}{\sqrt{\textnormal{vol}(M)}} , \frac{\varphi}{\sqrt{\textnormal{vol}(M)}} \right)_{L^2} ,
    $$
    which, due to 
    Lemma \ref{lemma:EigenFunc}, finally yields that $\lambda_1(A) = \lambda_1(\xi) = \norm{\xi - \omega}^2 = \textnormal{constant}$. 
\end{proof}

\subsection{The K--contact regular case} \label{subsec:circleBundle}

Computations for the bounds on $\lambda_1(A)$ simplify in the case of K--contact and regular manifolds. In the rest of this section, we pick $A \in \Omega^1(M)$ real-valued 
with $\diffRm \horiz{A} = 0$.

We start by listing some known facts about regular manifolds, that are mainly taken from \cite{Geiges2008,Blair2010}. A nowhere vanishing vector field $X \in \Gamma(TM)$ is called \textit{regular} if around each point of $M$ there is a flow box that is pierced at most once by any integral curve of $X$. We say that $M$ is \textit{regular} if its Reeb is a regular vector field. From now on, we assume $M$ is K--contact and regular. The Boothby--Wang theorem yields the existence of a constant $c \in \mathbb{R}$ and of a $2 n$--dimensional compact symplectic manifold $(B,\Omega)$, associated with an integral co-cycle $0 \neq [ \Omega ]_{\textnormal{hm}} \in H^2_{\textnormal{hm}}(B,\mathbb{Z})$, such that $M$ is the bundle space of a circle bundle $\pi : M \to B$ satisfying $\diff (c \, \alpha) = \pi^* \Omega$. In addition, the scaled Reeb $\Reeb / c$, associated with the scaled contact 1--form $c \, \alpha$, generates the free $\mathbb{S}^1$--action $\Theta : \mathbb{S}^1 \times M \to M$ on $\pi : M \to B$. From now on, without loss of generality, we denote $c \, \alpha$ with $\alpha$ and $\Reeb / c$ with $\Reeb$. Since $\Theta$ generates the flow of the Reeb and $\mathcal{L}_{\Reeb} \Jop = 0$ in K--contact settings, we infer that $J$ is right $\mathbb{S}^1$--invariant:
\begin{equation} \label{eq:JS1invariant}
    \Theta( \vartheta , \cdot )^* \Jop = \diff \Theta( \vartheta , \cdot ) \cdot \Jop , \quad \vartheta \in \mathbb{S}^1 .
\end{equation}
Moreover, the distribution $\Dist = \textnormal{ker} \ \alpha$ is a connection on $M$, thanks to which, for any $q \in B$ and $p \in \pi^{-1}(q)$, the horizontal lift $\widetilde v_p \in T_p M$ of any $v_q \in T_q B$ is uniquely well--defined such that $\widetilde v_p \in \Dist(p)$ and $\diff \pi(p) \cdot \widetilde v_p = v_q$. Also, $\widetilde v_p$ is right $\mathbb{S}^1$--invariant:
\begin{equation} \label{eq:liftS1invariant}
    \widetilde v_{\Theta( \vartheta , p )} = \diff \Theta( \vartheta , \cdot )(p) \cdot \widetilde v_p , \quad \vartheta \in \mathbb{S}^1 , \quad q \in B , \quad p \in \pi^{-1}(q) , \quad \textnormal{and} \quad v_q \in T_q B .
\end{equation}
Finally, since $\mathcal{L}_R \diff \alpha = 0$, we infer that also $\diff \alpha$ is right $\mathbb{S}^1$--invariant:
\begin{equation} \label{eq:dalphaS1invariant}
    \Theta(\vartheta,\cdot)^* \diff \alpha = \diff \alpha , \quad \vartheta \in \mathbb{S}^1 .
\end{equation}
Due to \eqref{eq:JS1invariant}--\eqref{eq:dalphaS1invariant}, we easily conclude that we may equip $B$ with the well--defined Riemannian metric $\metricB(q)(v_q,w_q) \triangleq \metric(p)(\widetilde v_p,\widetilde w_p)$ for $q \in B$ and $v_q , w_q \in T_q B$, and any $p \in \pi^{-1}(q)$. Its norm is denoted by $\normB{\cdot}$. Note that $\pi : (M,\metric) \to (B,\metricB)$ is a Riemannian submersion. We denote by $\scalarB{\cdot}{\cdot}: \Omega^{\bullet}(B) \times \Omega^{\bullet}(B) \to C^{\infty}(B,\mathbb{C})$ the scalar product on the external algebra of complex forms $\Omega^{\bullet}(B)$ induced by $\metricB$, and with no confusion arising, $\normB{\cdot}$ denotes the norm associated with $\scalarB{\cdot}{\cdot}$.

We have the following.

\begin{proposition} \label{proposition:Gysin}
    It holds that $H^1_{\textnormal{dR}}(B,\mathbb{R}) \overset{\pi^*}{\simeq} H^1_{\textnormal{dR}}(M,\mathbb{R})$. In particular, given $\omega \in \Omega^1(M)$ real-valued such that $\diffRm \horiz{\omega} = 0$, there exist $f , h \in C^{\infty}(M,\mathbb{R})$ and $\xi_B \in \Omega^1(B)$ harmonic and real-valued such that $\omega + \diff f + h \alpha = \pi^* \xi_B$. Moreover, the 1--form $\pi^* \xi_B$ is horizontal, basic, and itself harmonic.
\end{proposition}

\begin{proof}
    The first claim is classical. Below is the proof for the sake of completeness.

    Since $[\Omega]_{\textnormal{dR}} \neq 0$, we obtain that the Gysin long sequence 
    \begin{equation} \label{eq:GysinProof}
        \cdots \rightarrow H^1_{\textnormal{dR}}(B,\mathbb{R}) \xrightarrow{\pi^*} H^1_{\textnormal{dR}}(M,\mathbb{R}) \xrightarrow{\pi_*} H^0_{\textnormal{dR}}(B,\mathbb{R}) \xrightarrow{[\Omega]_{\textnormal{dR}} \wedge} H^2_{\textnormal{dR}}(B,\mathbb{R})\rightarrow \cdots
    \end{equation}
    is exact. On the one hand, $\pi^*$ is injective given that $\pi$ is a surjective submersion. On the other hand, since $[\Omega]_{\textnormal{dR}} \neq 0$ implies that $\textnormal{ker}([\Omega]_{\textnormal{dR}} \wedge \cdot) = \{ 0 \}$, we infer that $\pi^*$ is also surjective due to the exactness of the Gysin long sequence \eqref{eq:GysinProof}.

    For the second claim, Lemma \ref{lemma:reductionDiffRm} and the first claim yield the existence of $f , h \in C^{\infty}(M,\mathbb{R})$ and $\xi_B \in \Omega^1(B)$ closed and real-valued such that $A + \diff f + h \alpha = \pi^* \xi_B$. We note that we can select $\xi_B$ harmonic thanks to Hodge decomposition. Moreover, since $\Theta$ generates the flow of the Reeb, by leveraging any Darboux frame one readily checks that $\pi^* \xi_B$ is horizontal and basic. Finally, we conclude that $\pi^* \xi_B$ is itself harmonic from the fact that $\pi^*$ and $\delta$ commute on 1--forms, given that $\pi : (M,\metric) \to (B,\metricB)$ is a Riemannian submersion and 
    the fibres of $\pi$ are minimal \cite[Theorem 1.3]{Gilkey1999}.
\end{proof}

The following corollary of Theorem \ref{theo:bounds} states that, in the K--contact regular case, computing the upper bound \eqref{eq:newShortUpperBound} boils down to computing integrals on the basis.

\begin{corollary} \label{corol:BoundBasis}
    Assume that $A \in \Omega^1(M)$ satisfies $\diffRm \horiz{A} = 0$, and decompose $A + \diff f + h \alpha = \pi^* \xi_B$ as in Proposition \ref{proposition:Gysin}. It holds that
    \begin{equation} \label{eq:newShortUpperBoundBasis}
        \lambda_1(A) \le \frac{\textnormal{dist}(\pi^* \xi_B , \LHB^1(M))^2}{\textnormal{vol}(M)} = \frac{\textnormal{dist}(\xi_B , \Lfrak^1(B))^2}{\textnormal{vol}(B)} ,
    \end{equation}
    where
    \begin{align*}
        \Lfrak^1(B) \triangleq \Big\{ \omega \in \Omega^1(B) : \ &\omega \textnormal{\ is harmonic and} \int_{\gamma} \omega \in 2 \pi \mathbb{Z} , \\
        &\textnormal{for every AC and closed} \ \gamma : [0,1] \to B \Big\} .
    \end{align*}
    Moreover, if equality holds in bound \eqref{eq:newShortUpperBoundBasis}, then there exists $\omega_B \in \Lfrak^1(M)$ such that
    \begin{equation} \label{eq:newConstantForm}
        \lambda_1(A) = \normB{\xi_B - \omega_B}^2 = \textnormal{constant} , \quad \textnormal{pointwise on} \ B.
    \end{equation}
\end{corollary}

\begin{proof}
    We start by proving the equality in \eqref{eq:newShortUpperBoundBasis}, noticing that the inequality is straightforward consequence of Theorem \ref{theo:bounds}, Lemma \ref{lemma:Gauge}, and Proposition \ref{proposition:Gysin}.

    For this, we start noting that, thanks to Propositions \ref{prop:deRhamDec} and \ref{proposition:Gysin}, 
    given any $\omega \in \LHB^1(M)$ there exists $\omega_B \in \Omega^1(B)$ harmonic so that $\omega = \pi^* \omega_B$. If $\gamma_B : [0,1] \to B$ is any AC and closed curve, it is in particular a fibre map \cite[Theorem 5.2]{Husemoller1966}. Hence, there exists a lifting $\gamma : [0,1] \to M$ of $\gamma_B$ that is itself AC and closed, so that
    $$
    \int_{\gamma_B} \omega_B = \int_{\gamma} \pi^* \omega_B = \int_{\gamma} \omega \in 2 \pi \mathbb{Z} \quad \Longrightarrow \quad \omega_B \in \Lfrak^1(B) .
    $$
    Moreover, since $\pi : (M,\metric) \to (B,\metricB)$ is a Riemannian submersion, from the coarea formula we infer that, for every $f \in C^{\infty}(B,\mathbb{R})$,
    $$
    \int_M f \circ \pi = \lambda_{\Reeb} \int_B f , \quad \textnormal{where} \quad \lambda_{\Reeb} \triangleq \textnormal{``Reeb flow's period''} = \textnormal{constant} ,
    $$
    given that regularity forces $\textnormal{vol}( \pi^{-1}(q) ) = \lambda_{\Reeb}(p) = \textnormal{constant}$, for $q \in B$, $p \in \pi^{-1}(q)$. In turn, for any $\omega = \pi^* \omega_B \in \LHB^1(M)$ with $\omega_B \in \Lfrak^1(B)$, we may compute
    \begin{align*}
        \lambda_{\Reeb} \, \textnormal{dist}(\xi_B , \Lfrak^1(M))^2 \le \lambda_{\Reeb} \, \int_B \normB{\xi_B - \omega_B}^2 = \int_M \norm{\pi^* \xi_B - \omega}^2 ,
    \end{align*}
    which yields $\textnormal{dist}(\pi^* \xi_B , \LHB^1(M))^2 \ge \lambda_{\Reeb} \, \textnormal{dist}(\xi_B , \Lfrak^1(M))^2$ due to the arbitrariness of $\omega \in \LHB^1(M)$. The opposite inequality, thus also the fact that $\textnormal{vol}(M) = \lambda_{\Reeb} \, \textnormal{vol}(B)$, are proved similarly, and the equality in \eqref{eq:newShortUpperBoundBasis} follows.

    Finally, to prove \eqref{eq:newConstantForm}, let equality holds in \eqref{eq:newShortUpperBoundBasis}. From \eqref{eq:constantForm} we infer the existence of $\omega = \pi^* \omega_B \in \LHB^1(M)$ with $\omega_B \in \Lfrak^1(B)$ such that
    $$
    \lambda_1(A) = \norm{\pi^* \xi - \omega}^2(p) = \normB{\xi - \omega_B}^2 \circ \pi(p) , \quad \textnormal{for every} \quad p \in M ,
    $$
    and the conclusion follows.
\end{proof}

\section{The Three--dimensional K--contact Regular Case}
\label{sec:3D}
To further improve our bounds, we focus on three--dimensional, K--contact, and regular manifolds. Recall that a three--dimensional contact manifold is always a CR--manifold. Therefore, three--dimensional, K--contact, and regular manifolds are 
Sasakian.

\subsection{Landau quantization--type phenomenon} \label{sec:Landau}

In this section, we show that regularity in three–dimensional K--contact manifolds yields a \textit{Landau quantization-type phenomenon} for 
magnetic Laplacian, i.e., Theorem \ref{theo:Landau}. 

To best formalize this result, we start by listing some known facts about complex line bundles and magnetic Laplacians on Riemannian manifolds. Let $B$ denote a compact Riemannian surface equipped with a magnetic field $\Omega \in \Omega^2(B)$. Unless $B$ is simply connected, there may not exist a magnetic potential $A \in \Omega^1(B)$ satisfying $\diff A = \Omega$. However, if $\Omega$ is integral, that is if
$$
[ \Omega ]_{\textnormal{hm}} \in H^2_{\textnormal{hm}}(B,\mathbb{Z}) \quad \iff \quad \int_B \Omega \in \mathbb{Z} ,
$$
then there exists a complex line bundle $\mathfrak{p} : L \to B$ and a connection $\nabla$ with curvature $i \, 2 \pi \, \Omega$, i.e., with Chern class $[ \Omega ]_{\textnormal{hm}}$. That is, locally there exists a $1$--form $A \in \Omega^1(B)$ such that $\nabla = \diff + i A$, with $\diff A = i \, 2 \pi \, \Omega$. 
Hence, a magnetic Laplacian $\Delta^B_{\Omega}$ on $B$ can be defined through its action on the space of smooth sections $\Gamma(B,L)$ of $L$ as follows. For any open neighborhood $U \subseteq B$, select a local unitary frame $\sigma_U : U \to L$, so that any 
$s\in \Gamma(B,L)$ can be locally written as $s|_U = f \sigma$, for some $f\in C^\infty(U,\mathbb{C})$. The magnetic Laplacian $\Delta^B_{\Omega}$ may thus be defined by acting as $(\Delta^B_{\Omega} s)|_U = (\Delta_A f) \sigma$, where $\Delta_A$ is the magnetic Laplacian on $U$ associated with the magnetic potential $A|_U$.

At this step, let $M$ be a three--dimensional K--contact regular manifold, and $A \in \Omega^1(M)$ a given magnetic potential. Due to the discussion in Section \ref{subsec:circleBundle}, up to rescaling the contact 1--form by a constant, there exists a compact symplectic surface $(B,\Omega)$, associated with an integral co-cycle $0 \neq [ \Omega ]_{\textnormal{hm}} \in H^2_{\textnormal{hm}}(B,\mathbb{Z})$, such that $M$ is the bundle space of a circle bundle $\pi : M \to B$ satisfying $\diff \alpha = \pi^* \Omega$. In addition, the Reeb $\Reeb$ generates the free $\mathbb{S}^1$--action $\Theta : \mathbb{S}^1 \times M \to M$ on $\pi : M \to B$. By denoting $C^\infty(M,\mathbb{C})_k \triangleq \{ f \in C^\infty(M,\mathbb{C}) : \Reeb f = i k f \}$, we may decompose
\begin{equation}
    \label{eq:decomp-cinfty}
    C^\infty(M,\mathbb{C}) \simeq \bigoplus_{k\in \mathbb{Z}} C^\infty(M,\mathbb{C})_k .
\end{equation}
The condition $[ \Omega ]_{\textnormal{hm}} \in H^2_{\textnormal{hm}}(B,\mathbb{Z})$ ensures that each $C^\infty(M,\mathbb{C})_k$ is isomorphic to the space of smooth sections of the $k$--th tensor power of the complex line bundle $\mathfrak{p} : L \to B$ associated with the circle bundle $\pi : M \to B$.
Such an isomorphism is
\begin{equation} \label{eq:isomorphism}
    f \in C^\infty(M,\mathbb{C})_k \mapsto s_f \in \Gamma(B,L^{\otimes k}) , \quad \textnormal{with} \quad s_f(\pi(p)) \triangleq [p , f(p)]_{L^{\otimes k}} ,
\end{equation}
where $[p,z]_{L^{\otimes k}}$ denotes the equivalence class of $(p,z) \in M \times \mathbb{C}$ in $L^{\otimes k} = (M \times \mathbb{C}) / \sim$, with equivalence relation stemming from $(p,z) \sim (\Theta(\vartheta,p) , e^{-i k \vartheta} z)$, for any $\vartheta \in \mathbb{S}^1$. In particular, we may identify $C^\infty(M,\mathbb{C})_0 \simeq C^\infty(B,\mathbb{C})$.

Recall that $\DeltaRm$ denotes the horizontal Laplacian on $M$ (Section \ref{sec:DefLap}), generated by $\metric$. Since in K--contact settings $\DeltaRm$ and $\Reeb$ commute as differential operators, $\DeltaRm$ preserves the decomposition \eqref{eq:decomp-cinfty}, so that the spectrum of $\DeltaRm$ is the union of the spectra of each restriction $\DeltaRm|_k$ of $\DeltaRm$ on $C^\infty(M,\mathbb{C})_k$, for $k$ running over $\mathbb{Z}$. Note that, by leveraging the isomorphism \eqref{eq:isomorphism}, we may think of each $\DeltaRm|_k$ as an operator acting on $\Gamma(B,L^{\otimes k})$. In fact, each $\DeltaRm|_k$ results in a magnetic Laplacian on $B$, where the magnetic potential is given by a connection with curvature $i \, 2 \pi \, k \, \Omega$. In particular, through the results of Section \ref{subsec:circleBundle} and from the classical diamagnetic inequality, we infer that $\lambda_1(\DeltaRm|_k) \ge \lambda_1(\Delta^B)$, where $\lambda_1(\Delta^B)$ denotes the first non--zero eigenvalue of the classical Laplace-Beltrami operator $\Delta^B$ on $B$, generated by $\metricB$. Since $\DeltaRm|_0$ is isospectral to $\Delta^B$, from 
the inequality $\lambda_1(\DeltaRm|_k) \ge \lambda_1(\Delta^B)$ we conclude that
$$
\lambda_1(\DeltaRm) = \min_{k \in \mathbb{Z}} \; \lambda_1(\DeltaRm|_k) \le \lambda_1(\DeltaRm|_0) = \lambda_1(\Delta^B) \quad \Longrightarrow \quad \lambda_1(\DeltaRm) = \lambda_1(\Delta^B) .
$$

The latter identity does not hold in case of magnetic horizontal Laplacians $\DeltaRmA{A}$ on $M$, with a non--trivial magnetic potential $A\in \Omega^1(M)$, which we assume to be horizontal thanks to Lemma \ref{lemma:Gauge}. To see this, since the decomposition \eqref{eq:decomp-cinfty} still holds true, we denote by $\DeltaRmA{A}|_k$ the restriction of $\DeltaRmA{A}$ to $C^\infty(M,\mathbb{C})_k$. If we assume $A$ to be basic, then $\DeltaRmA{A}$ commutes with $R$, and thus $\DeltaRmA{A}|_k$ is well-defined. By again leveraging the isomorphism \eqref{eq:isomorphism}, we may think of each $\DeltaRmA{A}|_k$ as an operator acting on $\Gamma(B,L^{\otimes k})$. In fact, each $\DeltaRmA{A}|_k$ results in a magnetic Laplacian $\Delta^B_k$ on $B$, where the magnetic potential is given by a connection now with curvature $ i \, 2 \pi \, k \, \Omega + \diff A_B$. Here, $A_B \in \Omega^1(B)$ is the unique $1$--form on $B$ satisfying $A = \pi^* A_B$, which, since $\mathcal{L}_{\Reeb} A = 0$, can be well--defined by $A_B(q) \cdot v_q \triangleq A(p) \cdot \widetilde v_p$, for $q \in B$ and $v_q \in T B$, where $\widetilde v_p$ denotes the unique lift of $v_q$ at $p \in M$, for any $p \in \pi^{-1}(q)$. 
Now, denote by $\Delta^B_{A_B}$ the magnetic Laplacian on $B$ with magnetic potential $A_B$, generated by $\metricB$. Observe that, although it is 
clear that $\DeltaRmA{A}|_0$ is isospectral to $\Delta^B_{A_B}$, it is no more true that $\lambda_1(\DeltaRmA{A}|_k) \geq \lambda_1(\Delta^B_k)$, 
$k \neq 0$.

\begin{remark}
    Note that it is not possible to consider a magnetic field on $B$ (in the sense of complex line bundles as previously) generating a circle bundle $\pi : N \to B$ and an operator $\DeltaRm$ on some manifold $N$ with same spectrum as $\DeltaRmA{A}$ on $M$. This because the magnetic potentials $A_k$ constructed above do not scale correctly with $k$.
\end{remark}

From now on, for any $A \in \OmegaH^1(M)$ basic, let $A_B \in \Omega^1(B)$ denote the unique $1$--form on $B$ satisfying $A = \pi^* A_B$. Moreover, with obvious construction, let $\Delta^B_{A_B}$ denote the magnetic Laplacian on $B$ with magnetic potential $A_B$, generated by $\metricB$, and $\lambda_1(A_B)$ denote its first eigenvalue. To finally prove the aforementioned Landau quantization–type phenomenon, we start by recalling the following result (see, e.g., in \cite[Theorem~4]{montgomeryHearing1995}).

\begin{theorem} \label{theo:lambdaB}
    Let $B$ be an orientable Riemannian surface, and $\Omega \in \Omega^2(B)$ be an integral co-cycle.
    Let $\mathfrak{p} : L \to B$ be a line bundle whose connection has curvature $i \, 2 \pi \, \Omega$, and consider the associated magnetic Laplacian $\Delta^B_{\Omega}$ acting on the space of smooth sections $\Gamma(B,L)$ of $L$. Then, it holds that
    $$
    \lambda_1(\Delta^B_{\Omega}) \ge \mu(B,\Omega),
    \qquad\text{where}\qquad 
    \mu(B,\Omega)\triangleq \inf_{s\in\Gamma(B,L)} \frac{\displaystyle \left|\int_B |s|^2 \Omega\right|}{\|s\|^2_{L^2}}.
    $$
    If, in addition, $\Omega$ is a constant multiple of the Riemannian volume, we have that
    \[
    \lambda_1(\Delta^B_{\Omega}) = \mu(B,\Omega) = \frac{1}{\textnormal{vol}(B)}  \left| \int_B \Omega \right|.
    \]
\end{theorem}


We may now state our Landau quantization-type phenomenon as follows.

\begin{theorem} \label{theo:Landau}
    Let $A \in \OmegaH^1(M)$ be basic and $A_B \in \Omega^1(B)$ denote the unique $1$--form on $B$ satisfying $A = \pi^* A_B$. 
    Then, we have that
    \[
     \lambda_1(A_B) \ge  \lambda_1(A)
     \ge
     \min\left\{ \lambda_1(A_B) , \mu(B,\Omega)\right\} .
    \]
    Moreover, it holds that
    \[
        \lambda_1(A) = \lambda_1(A_B) 
        \qquad \text{if} \qquad
        \lambda_1(A_B) \le \mu(B,\Omega).
    \]
    Finally, if $\Omega$ is a constant multiple of the Riemannian volume, we have
    $$
    \lambda_1(A_B) \ge \lambda_1(A) = 
    \min\left\{ \lambda_1(A_B) , \ \frac{\displaystyle 1}{\displaystyle \textnormal{vol}(B)} \left|\displaystyle \int_B \Omega \right|\right\} .
    $$
\end{theorem}

\begin{proof}
    From the previous discussion we infer that $\lambda_1(\DeltaRmA{A}|_k) = \lambda_1(\Delta^B_{i k \Omega + \diff A_B})$. Since $[ i k \Omega + \diff A_B ]_{\textnormal{hm}} = [ i k \Omega ]_{\textnormal{hm}}$, Theorem \ref{theo:lambdaB} yields that
    $$
    \lambda_1(\DeltaRmA{A}|_k) = \lambda_1(\Delta^B_{i k \Omega + \diff A_B}) = 
    \mu(B,k\Omega) = |k|\mu(B,\Omega)\ge \mu(B,\Omega),
    \quad \forall k \in \mathbb{Z}\setminus\{0\} .
    $$
    We conclude from the fact that $\DeltaRmA{A}|_0$ is isospectral to $\Delta^B_{A_B}$.
    To complete the proof, it suffices to observe that in the case where $\Omega$ is a constant multiple of the Riemannian volume, we have $\lambda_1(\DeltaRmA{A}|_1)=\mu(B,\Omega)=(\operatorname{vol}(B))^{-1}|\int_B\Omega|$.
\end{proof}

\subsection{The case of Heisenberg nilmanifolds} \label{sec:Heisenberg}

Through the use of Theorem \ref{theo:Landau}, in this section we show how to explicitly compute the first eigenvalue in case of Heisenberg nilmanifolds.

Let $\mathbb{H}^1$ denote the Heisenberg group, i.e., the Lie group of upper triangular $3\times 3$ matrices with $1$s on their diagonal. In particular, we may identify $\mathbb{H}^1\simeq \mathbb{R}^3$ through the group law 
$(x,y,z) \cdot (x',y',z') = (x+x',y+y',z+z'+x'y)$, for $x , x' , y , y' , z , z' \in \mathbb{R}^3$. The 1--form 
$\alpha \triangleq \diff z - y \, \diff x$ endows $\mathbb{H}^1$ with a contact structure. The vector fields 
$X \triangleq \partial_x + y \partial_z$, 
$Y \triangleq \partial_y$, and $Z \triangleq \partial_z$ generate the algebra of left--invariant vector fields over $\mathbb{H}^1$. In particular, $\Dist = \operatorname{span}\{X,Y\}$, $Z$ is the Reeb, and the dual basis of $\{ X , Y , Z \}$ is $\{ \diff x , \diff y , \alpha \}$. The following result is well-known \cite[Theorem~5.4]{bockLowdimensional2016}.

\begin{theorem} \label{theo:LatticeH1}
    Every lattice (i.e., a discrete co-compact subgroup) of $\mathbb{H}^1$ is isomorphic to $\Gamma_k \triangleq \left\{ (k n_1 , n_2 , n_3) : \ n_1 , n_2 , n_3 \in \mathbb{Z} \right\}$, for some $k \in \mathbb{N}$, $k\ge 1$.
\end{theorem}

Due to Theorem \ref{theo:LatticeH1}, throughout this section we fix any integer $k \ge 1$ and consider the homogeneous space $M \triangleq \Gamma_k \setminus \mathbb{H}^1$. Since $\mathbb{H}^1$ acts freely and properly from the left over $\mathbb{H}^1$, $M$ is a compact three--dimensional manifold. Moreover, since $\alpha$ is left--invariant over $\mathbb{H}^1$, $M$ is naturally equipped with a (quotient) contact structure, whose corresponding contact 1--form will still be denoted by $\alpha$ (no confusion shall arise). With such contact structure, $M$ is called \textit{Heisenberg nilmanifold}. Since $X$, $Y$, and $Z$ are left-invariant, they naturally extend to vector fields over $M$, still denoted $X$, $Y$, and $Z$ through a slight abuse of notation. In particular, $\{ X , Y , Z \}$ results in a global basis for $T M$, with dual basis $\{ \diff x , \diff y , \alpha \}$ (no confusion shall arise with this latter notation). Endowed with the (1,1)--tensor $\Jop : TM \to TM$ satisfying $\Jop X = Y$, $\Jop Y = -X$, and $J Z = 0$, for which $\{ X , Y , Z \}$ is $\metric$--orthonormal, the contact manifold $M$ is K--contact. We refer to this latter (1,1)--tensor $\Jop$ as the \textit{Heisenberg almost--complex structure}.

Heisenberg nilmanifolds are regular. The circle bundle $\pi$ defined in Section \ref{subsec:circleBundle} can be computed explicitly. More specifically, denote the $k$--scaled two--dimensional torus by $\mathbb{T}^2_k \triangleq \mathbb{R} / k \mathbb{Z} \times \mathbb{R} / \mathbb{Z}$. With obvious notation, the surjective submersion $\pi : M \to \mathbb{T}^2_k$ defined by $\pi( [x , y , z]_{\Gamma_k} ) \triangleq ( [x]_{k \mathbb{Z}} , [y]_{\mathbb{Z}} )$ is a circle bundle, equipped with the free right $\mathbb{S}^1$--action $\Theta : \mathbb{S}^1 \times M \to M$ defined by $\Theta( [r]_{\mathbb{Z}} , [x , y , z]_{\Gamma_k} ) \triangleq [ (x,y,z) \cdot (0,0,r) ]_{\Gamma_k}$\footnote{From now on, we identify $\mathbb{S}^1 \simeq \mathbb{R} / \mathbb{Z}$.}. We may thus select $\Omega \triangleq \diff v \wedge \diff w$ as symplectic form on $\mathbb{T}^2_k$ such that $\diff \alpha = \pi^* \Omega$. We denote $\metricP \triangleq \metricB$ and note that, by definition of $\metric$, the metric $\metricP$ is actually flat. The following stems from applying Proposition \ref{proposition:Gysin} to Heisenberg nilmanifolds.

\begin{proposition} \label{prop:nilmanifold}
    It holds that $H^1_{\textnormal{dR}}(M) \simeq \mathbb{R}^2$. Therefore, the space of Rumin--harmonic horizontal 1--forms on $M$, which coincides with the space of harmonic (in the classical sense) horizontal 1--forms on $M$, is generated by $\{ \diff x , \diff y \}$.
\end{proposition}

\begin{proof}
    For the first claim, 
    use Proposition \ref{proposition:Gysin} and Poincar\'e's duality for homology.

    For the second claim, note that $H^1_{\textnormal{dR}}(M) \simeq \mathbb{R}^2$ and the fact that the Rumin cohomology is isomorphic to the de Rham cohomology imply that the space of Rumin--harmonic horizontal 1--forms on $M$ is at most of dimension two.  But the horizontal 1--forms $\diff x$ and $\diff y$ are linearly independent and Rumin--harmonic.
\end{proof}

At this step, to explicitly compute the first eigenvalue, let $A \in \Omega^1(M)$ be such that $\diffRm \horiz{A} = 0$. Thanks to Propositions \ref{prop:deRhamDec} and \ref{prop:nilmanifold} there exist $a , b \in \mathbb{R}$ and $f \in C^{\infty}(M,\mathbb{R})$ such that $\horiz{A} = ( a \; \diff x + b \; \diff y ) \oplus \diffH f$. Through the notation of Theorem \ref{theo:Landau} we may write $( a \; \diff x + b \; \diff y )_B = ( a \; \diff v + b \; \diff w )$, thus having the following.

\begin{corollary} \label{corol:equalityHeisenberg}
    Let $A \in \Omega^1(M)$ be such that $\diffRm \horiz{A} = 0$. Given $a , b \in \mathbb{R}$ and $f \in C^{\infty}(M,\mathbb{R})$ such that $\horiz{A} = ( a \; \diff x + b \; \diff y ) \oplus \diffH f$, it holds that
    \begin{equation}
        \lambda_1(A) = \min \Big\{ \textnormal{dist}(a , 2 \pi k^{-1} \mathbb{Z})^2 + \textnormal{dist}(b , 2 \pi \mathbb{Z})^2 , \ 1 \Big\} .
    \end{equation}
\end{corollary}

\begin{proof}
    Since $\displaystyle \textnormal{vol}(\mathbb{T}^2_k) = \hspace{-1ex} \int_{\mathbb{T}^2_k} \hspace{-1ex} \Omega = k$, and standard Fourier analysis readily yields that
    $$
    \lambda_1(a \; \diff v + b \; \diff w) = \frac{\textnormal{dist}(a \; \diff v + b \; \diff w , \mathfrak{L}^1(\mathbb{T}^2_k))^2}{\textnormal{vol}(\mathbb{T}^2_k)} ,
    $$
    see \cite[Theorem 4]{Colbois2017}, due to Lemma \ref{lemma:Gauge} and Theorem \ref{theo:Landau} we 
    need to prove that
    $$
    \mathfrak{L}^1(\mathbb{T}^2_k) = \Big\{ 2 \pi k^{-1} n_1 \; \diff v + 2 \pi n_2 \; \diff w : \ n_1 , n_2 \in \mathbb{Z} \Big\} .
    $$
    For this, due to Proposition \ref{prop:nilmanifold}, we just need to find admissible values for $c , d \in \mathbb{R}$ such that $\int_{\gamma} ( c \, \diff v + d \, \diff w ) \in 2 \pi \mathbb{Z}$ for every AC and closed $\gamma : [0,1] \to \mathbb{T}^2_k$. By selecting, for any $n_1 , n_2 \in \mathbb{Z}$, $\gamma_1(t) \triangleq ( [k n_1 t]_{k \mathbb{Z}} , [ 0 ]_{\mathbb{Z}} )$, for which $\dot{\gamma}_1 = k n_1 \partial_v$, and $\gamma_2(t) \triangleq ( [ 0 ]_{k \mathbb{Z}} , [ n_2 t ]_{\mathbb{Z}} )$, for which $\dot{\gamma}_2 = n_2 \partial_w$, respectively yields that
    $$
    \int_{\gamma_1} ( c \, \diff v + d \, \diff w ) = c k n_1 \quad \textnormal{and} \quad \int_{\gamma_2} ( c \, \diff v + d \, \diff w ) = d n_2 ,
    $$
    and the desired claim follows, given that $\gamma_1$ and $\gamma_2$ are generators for $\pi_1(\mathbb{T}^2_k)$.
\end{proof}

Interestingly, 
Corollary \ref{corol:equalityHeisenberg} enables explicitly computing the Chern class $k$, therefore recovering full information on the topology of $M$.

\begin{corollary} \label{corol:FindKHeisenbergNew}
    For any $\ell \in \mathbb{N}^*$, we have
    \begin{equation}
        \lambda_1(2 \pi \ell^{-1} \, \diff x) = 0 \quad \iff \quad \ell \text{ divides } k .
    \end{equation}
    In addition, the sequence $\{\lambda_1(2 \pi \ell^{-1} \, \diff x) : \ell\in \mathbb{N}^*\}$ completely determines $k$. 
\end{corollary}

\begin{proof}
    By Corollary~\ref{corol:equalityHeisenberg} we have 
    \begin{equation}
        \lambda_1\left( \frac{2 \pi}{\ell} \, \diff x \right) = \min\left\{ \operatorname{dist}\left( \frac{2 \pi}{\ell} , \frac{2 \pi}{k} \mathbb{Z} \right)^2 , 1 \right\} = \min\left\{ \frac{4 \pi^2}{k^2} \operatorname{frac}\left(\frac k \ell\right)^2 , 1 \right\} ,
    \end{equation}
    where $\operatorname{frac}(x)$ is the fractional part of any $x \in \mathbb{R}$, that is $\operatorname{frac}(x) \triangleq x - \lfloor x\rfloor$.
    The fact that $\operatorname{frac}(x)=0$ if and only if $x\in \mathbb{Z}$ concludes the proof.
\end{proof}

\subsection{The case of first Betti number two} \label{sec:regularSasakian}

If for Heisenberg nilmanifolds the first eigenvalue can be computed explicitly, in case of three--dimensional K--contact regular manifolds with first Betti number two a structural result can be shown. More specifically, if the upper bound \eqref{eq:newShortUpperBoundBasis} is attained, then 
$M$ is necessarily equipped with the Heisenberg almost--complex structure.

Let $M$ be a three--dimensional K--contact regular manifold with $b_1(M) = 2$. To best formalize our structural result, we start by constructing a strict contactomorphism between $M$ and an appropriate Heisenberg nilmanifold $\Gamma_k \setminus \mathbb{H}^1$, with $k \in \mathbb{N} \setminus \{ 0 \}$.

\begin{proposition} \label{prop:contactomorphism}
    Let $M$ be a three--dimensional K--contact regular manifold with contact 1--form $\alpha_M$ and first Betti number $b_1(M) = 2$. There exist $k \in \mathbb{N} \setminus \{ 0 \}$ and a strict contactomorphism $(M , \alpha_M) \simeq (\Gamma_k \setminus \mathbb{H}^1 , \alpha)$, where $\alpha \triangleq \diff z - y \diff x$.
\end{proposition}

\begin{proof}
    Via the notation of Section \ref{subsec:circleBundle}, up to rescaling the contact 1--form $\alpha_M$ and the Reeb $\Reeb_M$ of $M$ by a constant, there exist $k \in \mathbb{N}$ and a compact symplectic surface $(\Sigma_{\mathfrak{g}},\Omega)$ of genus $\mathfrak{g}$, associated with an integral co-cycle $k = [ \Omega ]_{\textnormal{hm}} \in H^2_{\textnormal{hm}}(\Sigma_{\mathfrak{g}},\mathbb{Z})$, such that $M$ is the bundle space of a circle bundle $\pi_M : M \to \Sigma_{\mathfrak{g}}$ satisfying $\diff \alpha_M = \pi^*_M \Omega$. Also, the Reeb generates the free $\mathbb{S}^1$--action $\Theta : \mathbb{S}^1 \times M \to M$ on $\pi : M \to \Sigma_{\mathfrak{g}}$, whose orbits has period one. Since $b_1(M) = 2$, due to Proposition \ref{proposition:Gysin} and Poincar\'e's duality for homology we have $k \neq 0$, and we may identify $\Sigma_{\mathfrak{g}} \simeq \mathbb{T}^2_k$ and $\Omega \simeq \diff v \wedge \diff w$. But circle bundles are uniquely classified via the Euler class of their basis, thus the Euler class $[ \Omega ]_{\textnormal{hm}} = k$ uniquely identifies $\pi_M : M \to \mathbb{T}^2_k$. In particular, since the mapping $\pi : \Gamma_k \setminus \mathbb{H}^1 \to \mathbb{T}^2_k$ constructed in Section \ref{sec:Heisenberg} is a circle bundle, we conclude that $M$ is diffeomorphic to $\Gamma_k \setminus \mathbb{H}^1$ through some bundle isomorphism. This latter isomorphism also enables equipping $\Gamma_k \setminus \mathbb{H}^1$ with a contact 1--form $\alpha_k$ that satisfies $\diff \alpha_k = \pi^* \Omega$, by appropriately manipulating $\alpha_M$, showing the previous diffeomorphism yields a strict contactomorphism $(M , \alpha_M) \simeq (\Gamma_k \setminus \mathbb{H}^1 , \alpha_k)$. Note that 
    due to the explicit expression of the $\mathbb{S}^1$--action $\Theta$ in Section \ref{sec:Heisenberg}, we infer that the Reeb of $\alpha_k$ is $Z \triangleq \partial_z$.

    Let us conclude by showing that $(\Gamma_k \setminus \mathbb{H}^1 , \alpha_k) \simeq (\Gamma_k \setminus \mathbb{H}^1 , \alpha)$ through strict contactomorphisms, with $\alpha \triangleq \diff z - y \, \diff x$ as defined in Section \ref{sec:Heisenberg}. For this, we observe that $\diff \alpha_k = \pi^* \Omega = \diff \alpha$, thus again Proposition \ref{proposition:Gysin} yields the existence of $\xi \in \Omega^1(\mathbb{T}^2_k)$ harmonic and $f \in C^{\infty}(\Gamma_k \setminus \mathbb{H}^1,\mathbb{R})$ such that $\alpha_k = \alpha + \pi^* \xi + \diff f$. Since $\alpha_k$ and $\alpha$ share the same Reeb $Z$ and $\pi^* \xi$ is horizontal, it holds that $\mathcal{L}_Z f = 0$, hence $f([x,y,z]_{\Gamma_k}) = f_{\mathbb{T}^2_k}([x]_{k \mathbb{Z}} , [y]_{\mathbb{R}})$ for some $f_{\mathbb{T}^2_k} \in C^{\infty}(\mathbb{T}^2_k,\mathbb{R})$. We infer that the mapping
    $$
    [x,y,z]_{\Gamma_k} \in \Gamma_k \setminus \mathbb{H}^1 \mapsto [x,y,z+f([x,y,z]_{\Gamma_k})] \in \Gamma_k \setminus \mathbb{H}^1
    $$
    generates a strict contactomorphism $(\Gamma_k \setminus \mathbb{H}^1 , \alpha_k) \simeq (\Gamma_k \setminus \mathbb{H}^1 , \alpha + \pi^* \xi)$, with $Z$ again as Reeb for $\alpha + \pi^* \xi$. At this step, define the smooth family of contact 1--forms $(\alpha_t)_{t \in [0,1]}$, where each $\alpha_t \triangleq \alpha + t \pi^* \xi$ has Reeb $Z$. Moser trick yields the existence of a smooth isotopy $(\psi_t)_{t \in [0,1]}$ of $\Gamma_k \setminus \mathbb{H}^1$ and a smooth family of smooth functions $(\lambda_t)_{t \in [0,1]}$, with $\lambda_t : \Gamma_k \setminus \mathbb{H}^1 \to \mathbb{R}_+$, such that $\psi^*_t \alpha_t = \lambda_t \alpha_0$ and $\frac{\diff}{\diff t}( \log \lambda_t ) \circ \psi^{-1}_t = \left( \frac{\diff}{\diff t} \alpha_t \right)(Z) = 0$ for every $t \in [0,1]$, see also \cite{Geiges2008}. In particular $\psi^*_1 (\alpha + \pi^* \xi) = \lambda_1 \alpha$, and since $\lambda_1 = \lambda_0 = 1$, 
    we finally infer that $(\Gamma_k \setminus \mathbb{H}^1 , \alpha_k) \simeq (\Gamma_k \setminus \mathbb{H}^1 , \alpha)$ through strict contactomorphisms.
\end{proof}

Due to Proposition \ref{prop:contactomorphism}, from now on we assume that $(M,\alpha_M) = (\Gamma_k \setminus \mathbb{H}^1 , \alpha)$. Through the pullback of a contactomorphism, the almost--complex structure $\Jop$ and the corresponding adapted Riemmanian metric $\metric$ on $M$ are naturally ``pulled--back'' on $\Gamma_k \setminus \mathbb{H}^1$. Thus, we may assume $\Gamma_k \setminus \mathbb{H}^1$ equipped with some almost--complex structure and corresponding adapted Riemmanian metric, still denoted by $\Jop$ and $\metric$ respectively. Similarly, any closed $A \in \Omega^1(M)$ can be naturally embedded on a closed 1--form on $\Gamma_k \setminus \mathbb{H}^1$. The main result of this section below states that, if the upper bound \eqref{eq:newShortUpperBoundBasis} is attained, then $\Jop$ must be the Heisenberg almost--complex structure.

\begin{theorem} \label{theo:equalityGeneral3DContact}
    Let 
    $A \in \Omega^1(M)$ be such that $\diffRm \horiz{A} = 0$. Through the above contact--based identifications on $\Gamma_k \setminus \mathbb{H}^1$, if equality holds in bound \eqref{eq:newShortUpperBound}, then 
    $M$ is necessarily equipped with the Heisenberg almost--complex structure defined in Section \ref{sec:Heisenberg}.
\end{theorem}

\begin{proof}
    Let us identify $(M,\alpha_M) = (\Gamma_k \setminus \mathbb{H}^1 , \alpha)$ as above. Consider any almost--complex structure $\Jop : TM \to TM$ and let $\metric$ be the associated adapted Riemannian metric. The construction of the circle bundle $\pi : M \to \mathbb{T}^2_k$ in Section \ref{sec:Heisenberg} and \eqref{eq:JS1invariant}--\eqref{eq:dalphaS1invariant} still hold, therefore we may equip $\mathbb{T}^2_k$ with the well--defined Riemannian metric $\metricP(q)(v_q,w_q) \triangleq \metric(p)(\widetilde v_p,\widetilde w_p)$ for $q \in \mathbb{T}^2_k$ and $v , w \in T_q \mathbb{T}^2_k$, and any $p \in \pi^{-1}(q)$. Note that, unlike the setting of Section \ref{sec:Heisenberg}, $\metricP$ is not a priori flat. Let us show that $\metricP$ is indeed flat 
    when equality is attained in the upper bound \eqref{eq:newShortUpperBound}.

    For this, due to Propositions \ref{proposition:Gysin} and \ref{prop:nilmanifold}, there exist $a , b \in \mathbb{R}$ and $f \in C^{\infty}(M,\mathbb{R})$ such that $\horiz{A} = a \; \diff x + b \; \diff y + \diffH f = \pi^*( a \; \diff v + b \; \diff w ) + \diffH f$. If equality holds in bound \eqref{eq:newShortUpperBound}, then thanks to Corollary \ref{corol:BoundBasis}, on the one hand we infer that
    \begin{equation} \label{eq:equalityLambda1withTorusProof}
        \lambda_1(A) = \frac{\textnormal{dist}(a \; \diff v + b \; \diff w , \mathfrak{L}^1(\mathbb{T}^2_k))^2}{\textnormal{vol}(\mathbb{T}^2_k)} , \quad \textnormal{with} \quad a \; \diff v + b \; \diff w \in \Omega^1(\mathbb{T}^2_k) ,
    \end{equation}
    and $\mathfrak{L}^1(\mathbb{T}^2_k)$ defined as in Corollary \ref{corol:BoundBasis}. On the other hand, Theorem \ref{theo:Landau} in particular yields that $\lambda_1(A) \le \lambda_1(a \; \diff v + b \; \diff w)$, with $\lambda_1(a \; \diff v + b \; \diff w)$ the first eigenvalue of the magnetic Laplacian on $\mathbb{T}^2_k$ as defined in Section \ref{sec:Landau}. By combining this latter inequality with \eqref{eq:equalityLambda1withTorusProof}, thanks to \cite[Theorem 3]{Colbois2017} we infer that
    $$
    \lambda_1(a \; \diff v + b \; \diff w) = \frac{\textnormal{dist}(a \; \diff v + b \; \diff w , \mathfrak{L}^1(\mathbb{T}^2_k))^2}{\textnormal{vol}(\mathbb{T}^2_k)} ,
    $$
    which, due to \cite[Theorem 4]{Colbois2017}, finally yields that the metric $\metricP$ must be flat.

    At this step, if $X , Y \in \Gamma(T M)$ are defined as in Section \ref{sec:Heisenberg}, from $X = \widetilde \partial_v$, $Y = \widetilde \partial_w$ we conclude that the  almost-contact structure $\Jop$ must satisfy $\Jop X = Y$ and $\Jop Y = -X$, thus it must be the Heisenberg almost--complex structure defined in Section \ref{sec:Heisenberg}.
\end{proof}

As natural and very interesting consequence, Corollary \ref{corol:FindKHeisenbergNew} and Theorem \ref{theo:equalityGeneral3DContact} enable computing the Chern class $k$ of $M$ in case $
\lambda_1(A)$ is known for a sufficiently rich family of magnetic potentials $A$. In particular, in this case we recover full information on $M$ and its contact structure. 
Indeed, the following holds true.

\begin{corollary}
    Let $M$ be a three--dimensional K--contact regular manifold with contact 1--form $\alpha_M$ and first Betti number $b_1(M) = 2$. In addition, let $A^1, A^2 \in \Omega^1(M)$ be linearly independent, Rumin--harmonic, and such that equality holds in bound \eqref{eq:newShortUpperBound} for either $A^1$ or $A^2$. Therefore, according to Theorem~\ref{theo:equalityGeneral3DContact}, $M$ is strictly contactomorphic to $(\Gamma_k\setminus\mathbb{H}^2,\alpha)$, where $\alpha\triangleq \diff z-y\diff y$. Moreover, there exists $\gamma_1,\gamma_2\in\mathbb{R}$ such that the Chern class $k \in \mathbb{N}\setminus\{0\}$ is uniquely determined by the sequence $\{ \lambda_1( ( \gamma_1 A^1 + \gamma_2 A^2 ) / \ell ) : \ell \in \mathbb{N}^* \}$.
\end{corollary}

\begin{proof}
    Due to Proposition \ref{prop:contactomorphism}, let $\mathfrak{c} : (M , \alpha_M) \to (\Gamma_k \setminus \mathbb{H}^1 , \alpha)$ be a contactomorphism. From Theorem \ref{theo:equalityGeneral3DContact} and our assumptions on $A_1$ and $A_2$, we infer that the almost--complex structure $J$ must be the Heisenberg almost--complex structure. Moreover, since $A^1$ and $A^2$ are linearly independent, it is always possible to find $\gamma_1,\gamma_2\in\mathbb{R}$ such that $\mathfrak{c}^* \left(\gamma_1\horiz{A^1} + \gamma_2 \horiz{A^2}\right) = 2\pi \, \diff x$. 
    The conclusion follows from Corollary~\ref{corol:FindKHeisenbergNew}.
\end{proof}

\appendix

\section{Proof of Lemma \ref{lemma:LaxMilgram}}
\label{sec:AppendixLaxMilgram}
We make use of the following lemma, which, due to the compactness of $M$, can be easily proved through a partition--of--unity argument.

\begin{lemma} \label{lemma:localW2}
    Given any neighborhood $U \subseteq M$, let $\{ \Reeb , E_1, \dots , E_{2n} \}$ denote a, e.g., Darboux frame defined in $U$. For $\varphi \in L^2(M,\mathbb{C})$, it holds that $\varphi \in \Wsecond(M,\mathbb{C})$ if $\varphi$ satisfies, for every $j,k=1,\dots,2n$ and neighborhood $U \subseteq M$,
    $$
    E_j \varphi \in L^2(U,\mathbb{C}) \ \ \textnormal{and} \ \ E_k E_j \varphi \in L^2(U,\mathbb{C}) \ \ \textnormal{(in the distributional sense)} .
    $$
\end{lemma}

\begin{proof}[Proof of Lemma \ref{lemma:LaxMilgram}]
    The Lax-Milgram Theorem yields that, for at least one real $\lambda > 0$, for every $f \in L^2(M,\mathbb{C})$ there uniquely exists $\varphi_f \in \Wfirst(M,\mathbb{C})$ such that
    \begin{equation} \label{eq:LaxProof}
        \int_M \metric(\nablaRmA{A} \varphi_f , \nablaRmA{A} \varphi) + \lambda \int_M \varphi_f \overline{\varphi} = \int_M f \overline{\varphi} , \ \textnormal{for every} \ \varphi \in \Wfirst(M,\mathbb{C}) .
    \end{equation}
    To prove $\textnormal{Rng}( \DeltaRmA{A} + \lambda I ) = L^2(M,\mathbb{C})$, it is sufficient to show that 
    $\varphi_f \in \Wsecond(M,\mathbb{C})$. For this, given any neighborhood $U \subseteq M$, let $\{ \Reeb , E_1, \dots , E_{2n} \}$ denote a, e.g., Darboux frame defined in $U$. From \eqref{eq:LaxProof} and Lemma \ref{lemma:MagnLap} we infer that (see also \cite{Franceschi2020})
    \begin{align*}
        \sum^{2n}_{j=1} \int_U \varphi_f &\overline{E^2_j \varphi + (\diver(E_j) - 2 i \horiz{A}(E_j)) E_j \varphi} = \\
        &= \int_U \Big( f + ( i \delta \horiz{A}  - \norm{\horiz{A}}^2 - \lambda ) \varphi_f \Big) \overline{\varphi} , \ \textnormal{for every} \ \varphi \in C^{\infty}_c(U,\mathbb{C}) .
    \end{align*}
    Since $M$ is a contact manifold, this latter variational problem falls into the hypoelliptic setting of \cite[Theorems 16 and 18]{Rothschild1976}, hence we infer that $E_j \varphi , E_k E_j \varphi \in L^2(U,\mathbb{C})$, for every $j,k=1,\dots,2n$, concluding that $\varphi_f \in \Wsecond(M,\mathbb{C})$ thanks to Lemma \ref{lemma:localW2}.

    To prove the second claim, we follows the arguments leveraged in \cite[Corollary B.3]{Chitour2024}. The first claim yields that $\DeltaRmA{A}$ is self--adjoint with $\sigma(\DeltaRmA{A}) \subseteq [0,+\infty)$. Hence, for every strictly positive $\lambda \in \mathbb{R}$, the resolvent
    $$
    R_{\lambda} \triangleq (\DeltaRmA{A} + \lambda I)^{-1} : \ L^2(M,\mathbb{C}) \to \Wsecond(M,\mathbb{C}) \subseteq L^2(M,\mathbb{C})
    $$
    is a well--defined linear and bounded operator. Therefore, for any $h \in C(M,(0,+\infty))$ and $\psi \in L^2(M,\mathbb{C})$, by denoting $\varphi \triangleq R_{\lambda} \psi \in W^2_H(M,\mathbb{C})$ and thanks to Lemma \ref{lemma:MagnLap} and Young's inequality we may additionally compute
    \begin{align*}
        \int_M \psi \overline{\varphi} &= \int_M \norm{\nablaRm \varphi}^2 + \int_M (\lambda + \norm{\horiz{A}}^2) | \varphi |^2  + 2 \ \textnormal{Im}\left( \int_M \scalar{\diffH \varphi}{\horiz{A} \varphi} \right) \\
        &\ge \int_M \norm{\nablaRm \varphi}^2 + \int_M (\lambda + \norm{\horiz{A}}^2) | \varphi |^2 - \int_M \frac{\norm{\nablaRm \varphi}^2}{h} - \int_M h \norm{\horiz{A}}^2 |\varphi|^2 .
    \end{align*}
    By taking $h \in C(M,(0,+\infty))$ defined, for $p \in M$, as $h(p) = 2$ when $\norm{\horiz{A}(p)}^2 < \lambda / 2$ and $h(p) = ( \norm{\horiz{A}(p)}^2 + \lambda / 2 ) / \norm{\horiz{A}(p)}^2$ otherwise, the latter inequality, together with the Cauchy--Schwartz inequality yield that
    \begin{align*}
        \min \left( \frac{1}{2} , \frac{\lambda}{2} \right) \| \varphi \|^2_{\Wfirst} \le \int_M \psi \overline{\varphi} \le \| \psi \|_{L^2} \| \varphi \|_{\Wfirst} ,
    \end{align*}
    from which it follows the existence of some $C_{\lambda} > 0$ such that $\| R_{\lambda} \psi \|_{\Wfirst} \le C_{\lambda} \| \psi \|_{L^2}$, for every $\psi \in L^2(M,\mathbb{C})$. Since the embedding $\Wfirst(M,\mathbb{C}) \hookrightarrow L^2(M,\mathbb{C})$ is compact \cite{Franceschi2020}, we conclude that $R_{\lambda} : L^2(M,\mathbb{C}) \to L^2(M,\mathbb{C})$ is compact, and therefore that $R_z : L^2(M,\mathbb{C}) \to L^2(M,\mathbb{C})$ is compact for every $z \in \rho(\DeltaRmA{A})$. We conclude from the fact that a self--adjoint operator with compact resolvent has discrete spectrum.
\end{proof}

\section{Global magnetic fields}
\label{sec:AppendixMagneticFields}
The definition given in this work of magnetic fields and potentials is local, in the sense that we consider magnetic potentials as 1--forms. Although this is the canonical approach in physics, and it is sufficient for our purposes, it is possible to give a somewhat ``more global'' definition of magnetic potentials and fields, which we briefly recall here.

Let $M$ be a contact manifold with contact distribution $\xi$ (here, we do not assume $M$ to be coorientable, therefore the contact form $\alpha$ may exist only locally). Consider a complex line bundle $\mathfrak{p} : L \to M$ endowed with some Hermitian metric $g$. Generally, a \emph{magnetic potential} is a partial connection $\nabla^\xi$ on $L$ along $\xi$. More specifically, the mapping $\nabla^\xi: \Gamma(L) \to \OmegaRm^1(M,L)\triangleq \OmegaRm^1(M)\otimes \Gamma(L)$ is $\mathbb{C}$-linear and satisfies the Leibniz rule 
$$
\nabla^\xi_X (f s) = f \nabla^\xi_X s + (X f) s , \quad f \in C^\infty(M,\mathbb{C}) , \ s \in \Gamma(L) , \ X \in \Gamma(\xi) .
$$
The \emph{horizontal magnetic Laplacian $\DeltaRmA{\nabla^\xi}$} associated with $\nabla^\xi$ may be defined as 
$$
s \in \Gamma(L) \mapsto \DeltaRmA{\nabla^\xi} s \triangleq (\nabla^\xi)^* \nabla^\xi s \in \Gamma(L) .
$$
This definition boils down to the one in Section \ref{sec:Preliminaries} in case $L$ is the trivial bundle $M\times \mathbb{C}$ endowed with the standard Hermitian metric, and $\nabla^\xi = {\diffH - i \horiz{A}}$ for some $A \in \Omega^1(M)$.

As in the canonical case, we would like to define a family of extensions $\nabla^\xi : \OmegaRm^k(M,L) \to \OmegaRm^{k+1}(M,L)$, i.e., to forms with values in $L$, by defining $\OmegaRm^k(M,L) \triangleq \OmegaRm^k(M) \otimes \Gamma(L)$ and
$$
\nabla^\xi(\omega \otimes s) = \diffRm \omega \otimes s + (-1)^k \omega \wedge \nabla^\xi s, \quad \omega \in \OmegaRm^k(M) , \  s \in \Gamma(L) .
$$
One easily checks that this does not work, as it would require that $\alpha\wedge \beta \in \OmegaRm^{k+1}(M)$ if $\alpha \in \OmegaRm^k(M)$ and $\beta \in \OmegaRm^1(M)$.
However, since $\nabla^\xi s\wedge \nabla^\xi s = 0$, the above expression still allows to define the \emph{curvature} of $\nabla^\xi$ as the map $F_{\nabla^\xi} : \Gamma(L) \to \OmegaRm^2(M,L)$ given by
$$
F_{\nabla^\xi} s = \nabla^\xi \circ \nabla^\xi s , \quad s \in \Gamma(L).
$$
Observe that, if locally $\nabla^\xi = \diffH - i \horiz{A}$ for some $A \in \Omega^1(M)$, then $F_{\nabla^\xi} = -i \diffRm \horiz{A}$. 



\bibliographystyle{amsalpha}
\bibliography{references}
\end{document}